\documentclass[11pt]{amsart}
 \usepackage{tikz}

\usepackage{listings}[2013/08/05]

\lstdefinelanguage{GAP}{%
  morekeywords={%
    Assert,Info,IsBound,QUIT,%
    TryNextMethod,Unbind,and,break,%
    continue,do,elif,%
    else,end,false,fi,for,%
    function,if,in,local,%
    mod,not,od,or,%
    quit,rec,repeat,return,%
    then,true,until,while%
  },%
  sensitive,%
  morecomment=[l]\#,%
  morestring=[b]",%
  morestring=[b]',%
}[keywords,comments,strings]

\usepackage[T1]{fontenc}
\usepackage[variablett]{lmodern}
\usepackage{xcolor}
\usepackage{lmodern}

\usepackage[all,cmtip]{xy}

\usepackage{geometry}
\usepackage{adjustbox}

\usepackage{hyperref}

\newcommand{\sk}{\smallskip}
\newcommand{\mk}{\medskip}
\newcommand{\bk}{\bigskip}

\usepackage{shuffle}

\usepackage{bigints}

\usepackage{tikz}
\makeatletter
\newcommand{\xleftrightarrow}[2][]{\ext@arrow 3359\leftrightarrowfill@{#1}{#2}}
\makeatother

\newcommand{\xdasharrow}[2][->]{
\tikz[baseline=-\the\dimexpr\fontdimen22\textfont2\relax]{
\node[anchor=south,font=\scriptsize, inner ysep=1.5pt,outer xsep=2.2pt](x){#2};
\draw[shorten <=3.4pt,shorten >=3.4pt,dashed,#1](x.south west)--(x.south east);
}
}

\usepackage{amsmath}
\usepackage{amssymb}
\usepackage{amsthm}
\usepackage{latexsym}
\usepackage{pstricks}
\usepackage{amsbsy}
\usepackage{amscd}
\usepackage{mathrsfs}
\usepackage{tipa}
\usepackage{txfonts}
\usepackage{amsfonts}
\usepackage{graphicx}
\usepackage{tabularx}
\usepackage{listings}
\usepackage{tikz}
\usepackage{hyperref}

\usepackage{scalerel,stackengine}
\stackMath
\newcommand\reallywidehat[1]{%
\savestack{\tmpbox}{\stretchto{%
  \scaleto{%
    \scalerel*[\widthof{\ensuremath{#1}}]{\kern-.6pt\bigwedge\kern-.6pt}%
    {\rule[-\textheight/2]{1ex}{\textheight}}%WIDTH-LIMITED BIG WEDGE
  }{\textheight}% 
}{0.5ex}}%
\stackon[1pt]{#1}{\tmpbox}%
}
\newtheorem{thm}{Theorem}[section]

\newtheorem{lem}[thm]{Lemma}

\newtheorem{prop}[thm]{Proposition}

\newtheorem{defn}[thm]{Definition}

\newtheorem{defn-prop}[thm]{Definition-Proposition}

 \newcommand{\eq}[1][r]
   {\ar@<-3pt>@{-}[#1]
    \ar@<-1pt>@{}[#1]|<{}="gauche"
    \ar@<+0pt>@{}[#1]|-{}="milieu"
    \ar@<+1pt>@{}[#1]|>{}="droite"
    \ar@/^2pt/@{-}"gauche";"milieu"
    \ar@/_2pt/@{-}"milieu";"droite"}

 \newcommand{\incl}[1][r]
  {\ar@<-0.2pc>@{^(-}[#1] \ar@<+0.2pc>@{-}[#1]}

\author[A.-M. Castravet and L. Pirio]{\href{mailto:ana-maria.castravet@uvsq.fr}{Ana-Maria Castravet}  \& \href{mailto:luc.pirio@uvsq.fr}{Luc Pirio}}
\title[Hyperlogarithmic functional equations on del Pezzo surfaces]
{Hyperlogarithmic functional equations on del Pezzo surfaces}

\begin{document}

\maketitle

%%%%%%%%%%%%%%%%%%%%%%%%%%%%%%%%%%%%%%%
%%%%%%%%%%%%%%%%%%%%%%%%%%%%%%%%%%%%%%%
%\begin{center} 
%{\Large \bf  Hyperlogarithmic functional identities on del Pezzo surfaces}
%%%%%%%%%%%%%%%%%%%%%%%%%%%%%%%%%%%%%%%
%%%%%%%%%%%%%%%%%%%%%%%%%%%%%%%%%%%%%%%
%\bigskip
%{\large Ana-Maria Castravet,  Luc Pirio${}^\dagger$}
%{\large Ana-Maria Castravet,  Luc Pirio\footnote{Coresponding author}}
%\bigskip
%{}
%\bigskip
%\today
%\end{center}

\newcommand{\homeo}{\textup{Homeo}^+(S,M)}
\newcommand{\homeoo}{\textup{Homeo}_0(S,M)}
\newcommand{\mg}{\mathcal{MG}(S,M)}
\newcommand{\mmg}{\mathcal{MG}_{\bowtie}(S,M)}
\newcommand{\Z}{\mathbf{Z}}
\newcommand{\Q}{\mathbf{Q}}
\newcommand{\N}{\mathbf{N}}

%%%%%%%%%%%%%%%%%%%%%%
%%%%%%%%%%%%%%%%%%%%%%

% ANA-MARIA'S MACROS

\newcommand{\CC}{\mathbf{C}}
\newcommand{\PP}{\mathbf{P}}
\newcommand{\ZZ}{\mathbf{Z}}

\def\Bl{\text{\rm Bl}}
\def\Pic{\text{\rm Pic}}

\newcommand{\Hh}{\text{H}}
\newcommand{\HH}{\boldsymbol{\mathcal H}}

\newcommand{\mc}{\mathfrak{c}}

\newcommand{\cK}{\mathcal{K}}
\newcommand{\cL}{\mathcal{L}}
\newcommand{\cO}{\mathcal{O}}
\newcommand{\cU}{\mathcal{U}}
\newcommand{\cV}{\mathcal{V}}

\def\hra{\hookrightarrow}
\def\ra{\rightarrow}

%%%%%%%%%%%%%%%%%%%%%%
%%%%%%%%%%%%%%%%%%%%%%

\hyphenation{ap-pro-xi-ma-tion}

\vspace{-1cm}
\begin{abstract}  For any $d\in \{1,\ldots,6\}$, we prove that the web of conics on a del Pezzo surface of degree $d$
carries a functional identity whose components are antisymmetric hyperlogarithms of weight $7-d$. 
Our approach is uniform with respect to $d$ and  relies on classical results about the action of the Weyl group on the set of lines on the del Pezzo surface. 
These hyperlogarithmic functional identities are natural generalizations of the classical  3-term and (Abel's) 5-term identities satisfied by the logarithm and the dilogarithm, which   correspond to the cases when $d=6$ and $d=5$ respectively.
\end{abstract}

\vspace{-0.3cm}
\section{Introduction}

%%%%%%%%%%%%%%%%%%%%%%%%%%%%%%%%%%%%%%%
\subsection{Functional equations of polylogarithms}
%%%%%%%%%%%%%%%%%%%%%%%%%%%%%%%%%%%%%%%
The classical logarithm ${\rm Log}$ satisfies  \emph{Cauchy's identity}
%%%%%%%%%%%%%%%%%%%%%%%%%%%%%%%%%%%%%%%%%%
\begin{equation}
\label{Eq:EFA-3-terms-Log}
%%%%%%%%%%%%%%%%%%%%%%%%%%%%%%%%%%%%%%%%%%
{\rm Log}(x)+{\rm Log}(y)-{\rm Log}(xy) = 0
\end{equation}
%`Cauchy identity' 
%%%%%%%%%%%%%%%%%%%%%%%%%%%%%%%%%%%%%%%%%%
for all $x,y>0$, and this functional identity is fundamental in mathematics. 

Several authors of the XIXth century have independently discovered equivalent forms of the following identity
%{\red (Equivalent forms of the following identity have been discovered in the  XIXth century:)}
$$
\boldsymbol{\big(\mathcal Ab\big)}
\hspace{3cm}
R\big(x\big)-R\big(y\big)-R\bigg(\frac{x}{y}\bigg)-R\left(\frac{1-y}{1-x}\right)
+R\left(\frac{x(1-y)}{y(1-x)}\right)=0\, , 
%\,\quad x,y\in \mathbf R,\quad 0<x<y<1
\hspace{3cm} {}^{}
$$
satisfied for all $x,y$ such that $0<x<y<1$, 
where $R$ stands for  \href{https://mathworld.wolfram.com/RogersL-Function.html}{\it Rogers' dilogarithm}, defined by
\begin{equation}
\label{Eq:R}
%R(x)= {\bf L}{\rm i}_2(x) + \frac{1}{2}\log(x)\log(1 - x) - %\frac
%{\pi^2}/{6}\, %\footnotemark
R(x)= {\bf L}{\rm i}_2(x) + \frac{1}{2}\,{\rm Log}(x)\,{\rm Log}(1 - x) - \frac
{{}^{}\hspace{0.1cm} \pi^2}{6}%\footnotemark
  \end{equation}
 for $x\in (0,1)$, 
where ${\bf L}{\rm i}_2$ stands for the classical bilogarithm, the weight 2 polylogarithm. 

The identity $\boldsymbol{\big(\mathcal Ab\big)}$ is nowadays called \emph{Abel's identity} of the dilogarithm, hence the notation.  
 It can be seen as a weight 2 generalization of 
 Cauchy's identity \eqref{Eq:EFA-3-terms-Log}.  To see in which way, recall the \emph{weight $n\geq 1$ polylogarithm ${\bf L}{\rm i}_n$}, 
classically   
defined on the unit disk ${\bf D}=\{\, z\in \mathbf C\,, \, \lvert z\lvert<1\, \}$ 
as the sum of the convergent series
% ${\bf L}{\rm i}_n(z)=\sum_{k\geq 1} z^k/k^n$.
   $${\bf L}{\rm i}_n(z)=\sum_{k\geq 1} z^k/k^n\, .$$    The first polylogarithm is related to the usual logarithm through the relation ${\bf L}{\rm i}_1(z)=-{\rm Log}(1-z)$ for $z\in \mathbf D$, and both ${\rm Log}$ and $R$ can be considered as 
   suitable versions of the first two polylogarithms such that the two functional identities \eqref{Eq:EFA-3-terms-Log}
 and $\boldsymbol{\big(\mathcal Ab\big)}$ hold true. 
 
Polylogarithms are special functions of great interest which satisfy properties  
 generalizing those of the logarithm and the dilogarithm. In particular, they satisfy functional equations of the form 
\begin{equation}
\label{Eq:Lin-equation}
\sum_{i=1}^{M} c_i \, {\bf L}{\rm i}_n(U_i)={\rm L}_{n-1}
\end{equation}
where the $c_i$'s are rational coefficients, the $U_i's$ are multi-variable rational functions and with ${\rm L}_{n-1}$ a rational expression in polylogarithmic functions of weight at most $n-1$.

Since the early XIX-th century (works of \href{https://arts.st-andrews.ac.uk/digitalhumanities/fedora/repository/islandora%3A2864#page/7/mode/2up}{Spence},  \href{http://www.digizeitschriften.de/dms/img/?PID=GDZPPN002142422}{Kummer},  \href{https://ia800204.us.archive.org/4/items/oeuvrescomplte02abel/oeuvrescomplte02abel.pdf}{Abel}, etc) 
 to nowadays, many authors have discovered functional identities of the above form satisfied by (some version of) polylogarithms ${\bf L}{\rm i}_n$ of weight $n\leq 7$.  
Multi-variable generalizations of polylogarithms have been considered as well, in particular their functional equations.  The subject is currently very active.\footnote{
{E.g.}, see the works \href{https://doi.org/10.1006/aima.1995.1045}{Goncharov}, \href{https://doi.org/10.1007/s00029-003-0312-z}{Gangl}, \href{https://arxiv.org/abs/1803.08585}{Goncharov-Rudenko}, \href{https://arxiv.org/abs/2012.09840}{Charlton-Gangl-Radchenko}, \href{https://arxiv.org/abs/2012.05599}{Rudenko}.} 
 Polylogarithms are connected to several distinct fields in mathematics such as hyperbolic geometry (volumes of hyperbolic polytopes), $K$-theory of number fields (Zagier's conjecture), theory of periods and multizeta values, scattering amplitudes in higher energy physics, theory of cluster algebras, etc.\footnote{ For more details, we refer to the surveys \cite{Zagier1989} or \cite{GoncharovICM}.} 
 %These papers are a bit outdated (they  are more than 30 years old in their original versions)  but are still very interesting to read.
 In particular, it is now clearly established that 
knowing functional identities of the form \eqref{Eq:Lin-equation} is important, 
cf.\,\cite[\S2]{GoncharovGelfandSeminar}. However, in spite of the important number of recent works on the subject, the functional identities satisfied by polylogarithms are still not well understood.

Identities of the form \eqref{Eq:Lin-equation} %satisfied by classical polylogarithms   
 are known to exist only for $n\leq 7$ and for the higher weights ($n=6, 7$) were obtained by computer aided 
 calculations (see \cite{Gangl}). The general belief is 
that, for any  $n\geq 1$, there should exist a fundamental identity of the form \eqref{Eq:Lin-equation} satisfied by ${\bf L}{\rm i}_n$ from which any other could be formally obtained (for instance, see the last paragraph of \cite[\S4.1]{Griffiths}).  
The first interesting case to be considered is for weight $2$,  
for which Abel's identity 
%$\boldsymbol{\big(\mathcal Ab\big)}$ 
 is the evoked fundamental one, a result which has been proved only recently in \cite{dJ}. For weight $3$,  the 22-term trilogarithmic equation in three variables found by Goncharov in \cite{Goncharov1995} may be the fundamental one but,  
as far we are aware of, there is no proof until now. The weight $4$ case is the subject of the recent  work \cite{GR} by Goncharov and Rudenko. Using 
the cluster structure of the moduli spaces $\mathcal M_{0,n+3}$, they construct a functional identity for the tetralogarithm which allows them to prove Zagier's conjecture in weight $4$. This identity is expected to play the same role for the tetralogarithm  as the one played by Abel's identity for the dilogarithm.

%%%%%%%%%%%%%%%%%%%%%%%%%%%%%%%%%%%%%%%%%
%%%%%%%%%%%%%%%%%%%%%%%%%%%%%%%%%%%%%%%%%
\subsection{Hyperlogarithms}
%%%%%%%%%%%%%%%%%%%%%%%%%%%%%%%%%%%%%%%%%
%%%%%%%%%%%%%%%%%%%%%%%%%%%%%%%%%%%%%%%%%
%%%%%%%%%%%%%%%%%%%%%%%%%%%%%%%%%%%%%%%%%
Hyperlogarithms are generalization of polylogarithms and they go back to Poincar\'e.  
They are multivalued holomorphic functions on $\mathbf P^1$ which can be obtained by iterated integrations of some given rational 1-forms with logarithmic singularities on the Riemann sphere. 
%$\mathbf P^1$.  
 More precisely, let 
$\sigma_1,\ldots,\sigma_{m+1}$ be $m+1$ pairwise
distinct points of $\PP^1$. We fix an affine coordinate  $z$ such that $\sigma_{m+1}=\infty$. Then the 1-forms $\omega_s=dz/(z-\sigma_s)$ for $s=1,\ldots,m$ form a basis of the space 
%
%${\bf H}^0\big(
%\mathbf P^1, \Omega^1_{\mathbf P^1} ( {\rm Log}\,\Sigma) 
%\big)$
 of global logarithmic 1-forms on $\PP^1$ with poles in $\Sigma=\{ \sigma_s\}_{s=1}^{m+1}$.  We set $Z=\mathbf P^1\setminus \Sigma$. 
 
Given any tuple $( s_k)_{k=1}^w$ of elements in $\{1,\ldots,m\}$, the \emph{weight $w$ hyperlogarithm $L_{\omega_{s_1}\cdots \omega_{s_w}}$} is the multivalued 
function on $Z$ defined inductively as follows: 
%by successive integrations according to the following relations: 
$$
L_{\omega_{s_w}}(z)=\int^z \omega_{s_w}= {\rm Log}\big(z-\sigma_{s_w}\big)
\qquad 
\mbox{ and } 
\qquad 
L_{ \omega_{s_1}\ldots \omega_{s_w} } (z)= 
%\int^{z}
%\omega_{s_1}\otimes \cdots \otimes \omega_{s_w}  = 
\int^{z} \frac{ L_{ \omega_{s_2}\ldots \omega_{s_w} } (u)}{u-\sigma_{s_1}} du
$$

The polylogarithmic functions (such as ${\rm Log}$, Rogers' dilogarithm $R$, or all the classical polylogarithm ${\bf L}{\rm i}_n$) are particular instances of hyperlogarithms in the specific case when  $\Sigma=\{0,1,\infty\}$.  If the properties of polylogarithms, in particular the functional equations they satisfy, have been studied intensely, this is 
much less the case for more general hyperlogarithms (however see \cite{W1} or the more recent \cite{Brown}).  
Several recent works have shown that 
hyperlogarithms are relevant for computing  
certain scattering amplitudes in higher energy physics (see for instance the PhD thesis \cite{Panzer}  or the recent ``white paper" \cite{white}, especially the fifth section therein).
\sk

In this paper, we describe generalizations in weight $3$, $4$, $5$ and $6$  of 
the $3$-term and $5$-term identities of the logarithm and dilogarithm respectively.
%$(\mathcal Ab\big)$
These identities are similar to the two latter classical identities, but involve non polylogarithmic hyperlogarithms.  
For this purpose, we  introduce a geometric viewpoint on Abel's identity 
$\boldsymbol{\big(\mathcal Ab\big)}$ by relating it to the conic fibrations of a quintic del Pezzo surface. 
The generalization will involve the conic fibrations of del Pezzo surfaces of 
degree $\leq 6$.\footnote{ The  3-term identity of the logarithm can also be considered from a geometric perspective, but it is less meaningful from this point of view, because it is ``too simple". This is why we only consider the case of Abel' equation.}  
\black

%%%%%%%%%%%%%%%%%%%%%%%%%%%%%%%%%%%%%%%
\subsection{Abel's identity on the quintic del Pezzo surface}
%%%%%%%%%%%%%%%%%%%%%%%%%%%%%%%%%%%%%%%
The five rational arguments  
$$U_1=x\, , \qquad U_2=y\, , \qquad   U_3=\frac{x}{y}
\, , \qquad 
U_4=\frac{1-y}{1-x} \qquad \mbox { and } \qquad 
U_5=\frac{x(1-y)}{y(1-x)}\,   $$
of $R$ in $\boldsymbol{\big(\mathcal Ab\big)}$ can be interpreted geometrically as follows: let 
$$
\beta: X_{4}={\bf Bl}_{p_1,\ldots,p_4}(\mathbf P^2) 
\longrightarrow \mathbf P^2 
$$
%$\beta: X_4\rightarrow \PP^2$ 
 be the blow-up of the complex projective plane at the $4$ points in general position $p_1=[1:0:0]$, $p_2=[0:1:0]$, $p_3=[0:0:1]$ and $p_4=[1:1:1]$. The surface is  the \emph{quintic del Pezzo surface}.  
 It carries five fibrations in conics $\phi_i: X_4\rightarrow \PP^1$ ($i=1,\ldots,5$) which coincide
with the compositions $U_i\circ \beta: X_4\dashrightarrow \PP^1$ 
as rational functions.  It follows that Abel's identity  can be written
$$
\boldsymbol{\big(\mathcal Ab_{X_4}\big)}
\hspace{5cm}
\sum_{i=1}^5 \epsilon_i \,R\big(\phi_i\big)=0 
\hspace{6cm} {}^{}
$$
for some constants $\epsilon_1,\ldots,\epsilon_5$ equal to $1$ or $-1$, this identity holding true locally at any sufficiently general point of $X_4$ for suitable branches of Rogers dilogarithm.

%%%%%%%%%%%%%%%%%%%%%%%%%%%%%%%%%%%%%%%
%%%%%%%%%%%%%%%%%%%%%%%%%%%%%%%%%%%%%%%
\subsection{Main result: generalization to del Pezzo surfaces of degree $\leq 6$} 
%%%%%%%%%%%%%%%%%%%%%%%%%%%%%%%%%%%%%%%
%%%%%%%%%%%%%%%%%%%%%%%%%%%%%%%%%%%%%%%
Let $3\leq r\leq8$ and let 
$$X_{r}={\bf Bl}_{p_1,\ldots,p_r}(\mathbf P^2)$$ be the blow-up of the projective plane at $r$ points in general position.  
Then $X_r$ is a del Pezzo surface of degree $9-r$, i.e., the anti-canonical 
class 
$-K_{X_r}$ is ample and with self-intersection 
$(-K_{X_r})^2=9-r$. If $3\leq r\leq 6$, then the complete linear system $-K_{X_r}$ defines an embedding $X_r\hra\PP^{9-r}$ such that the degree of $X_r$ is 
$9-r$. We define the \emph{degree} of a curve $C\subset X_r$ to be $C\cdot (-K_{X_r})$. 
Smooth rational curves in $X_r$ of degree $1$, respectively $2$, are called \emph{lines}, respectively \emph{smooth conics}. 
A conic fibration on $X_r$ is the equivalence class (up to post composition  with an element of $\rm{PGL_2}$) of a morphism $X_r\rightarrow \mathbf P^1$ 
such that a general fiber is a smooth conic. 

The following facts are well known: 
\vspace{-0.18cm}
\begin{itemize}
\item[${(i).}$]  The number $l_r$ of lines in $X_r$ is finite; \sk

\item[${(ii).}$]  The number $\kappa_r$ of conic fibrations on $X_r$ is finite as well;\sk 
\item[${(iii).}$]   Any conic fibration $X_r\rightarrow \mathbf P^1$ 
has exactly $r-1$ reducible fibers, each a union of two lines in $X_r$ intersecting transversely at one point.
\sk
\item[${(iv).}$] The Picard group $\Pic(X_r)$ is free  and is acted upon by a certain Weyl group $W_r$. Moreover, this action preserves the intersection product.  
\end{itemize}
\vspace{-0.1cm}
The values of $l_r$ and $\kappa_r$ for $3\leq r\leq 8$ are given in the following table: 
%%%%%%%%%%%%%%%%%%%%%%%%%%%%%%%%%%%%%%%%%%%
%%%%%%%%%%%%%%%%%%%%%%%%%%%%%%%%%%%%%%%%%%%
$$ \begin{tabular}{|c||c|c|c|c|c|c|}
%%%%%%%%%%%%%%%%%%%%%%%%%%%%%%%%%%%%%%%%%%%
\hline
$\boldsymbol{r}$ & 3 & 4& 5  & 6 & 7 & 8\\ \hline \hline
$\boldsymbol{l_r}$ & 6 & 10 &  16  & 27 & 56  & 240  \\ \hline
$\boldsymbol{\kappa_r}$ & 3 & 5 &    10 & 27 & 126 & 2160 \\
\hline
\end{tabular}
$$
%\caption{tt}
%\end{table}
%%%%%%%%%%%%%%%%%%%%%%%%%%%%%%%%%%%%%%%%%%%
%%%%%%%%%%%%%%%%%%%%%%%%%%%%%%%%%%%%%%%%%%%

Let $\phi_1,\ldots,\phi_{\kappa_r}: X_r\rightarrow \mathbf P^1$ be $\kappa_r$ pairwise non equivalent conic fibrations.  We denote by $L_r$ the divisor of $X_r$ whose the irreducible components are all the lines in $X_r$ and we set 
$$Y_r=X_r\setminus L_r\, .$$   
From ${(iii)}$, 
 we know that 
the complement $\Sigma_i$ of $\phi_i(Y_r)$ in $\mathbf P^1$ is a finite set with $r-1$ elements denoted by $\sigma_i^{1},\ldots,\sigma_i^{r-1}$.  One assumes that $\phi_i$ 
has been chosen such that one of the $\sigma_i^{t}$'s, say $\sigma_i^{r-1}$, coincides with $\infty\in \mathbf P^1$.  Then the  rational differentials $\omega_i^{t}= dz/(z-\sigma_i^{t})$ for $t=1,\ldots,r-2$ form a basis of the 
space of logarithmic 1-forms on $\mathbf P^1$ with poles along $\Sigma_i$.

 For all $i=1,\ldots,\kappa_r$,  let $AI_i^{r-2}$ be 
the \emph{complete antisymmetric hyperlogarithm} of weight $r-2$ on $Z_i=\mathbf P^1\setminus \Sigma_i$, defined as the antisymmetrization 
of the hyperlogarithm
$L_{\omega_i^1\cdots \omega_i^{r-2}}$ 
with respect to the logarithmic 1-forms $\omega_i^1,\ldots, \omega_i^{r-2}$,  i.e., 
$$
AI_i^{r-2}={\rm Asym}\Big(
L_{\omega_i^1\cdots \omega_i^{r-2}}
  \Big) = \frac{1}{(r-2)!}
\sum_{\nu \in \mathfrak S_{r-2}} (-1)^\nu \, 
L_{\omega_i^{\nu(1)}\cdots \omega_i^{\nu(r-2)}}\, 
$$
where, for any $\nu \in \mathfrak S_{r-2}$, we denote by $(-1)^\nu$ the signature of $\nu$.  Each  $AI_i^{r-2}$ is uniquely defined up to sign. 
Our main result is the following: 

%%%%%%%%%%%%%%%%%%%%%%%%%%%%%%%%%%%%%%%
\begin{thm}
\label{Thm:MMain}
%%%%%%%%%%%%%%%%%%%%%%%%%%%%%%%%%%%%%%%
There exists 
%a $\kappa_r$-tuple 
 $(\epsilon_i)_{i=1}^{\kappa_r}\in \{ \pm 1 \}^{\kappa_r}$, 
 unique up to a global sign, 
  such that 
for any $y\in Y_r$ and for a suitable choice of the branch of the hyperlogarithm $AI_i^{r-2}$ at $y_i=\phi_i(y)\in \mathbf P^1$ for each $i=1,\ldots,\kappa_r$, the following functional identity 
holds true  on an open neighbourhood of $y$ in $ Y_r$: 
$$ {}^{} \hspace{-6.3cm}
\boldsymbol{{\bf H Log}(X_r)}
\hspace{3.9cm}
\sum_{i=1}^{\kappa_r} \epsilon_i \,AI_i^{r-2}\big(\phi_i\big)=0\, .
$$
%holds true identically on a neighborhood of $y$ in $ Y_r$. 
%%%%%%%%%%%%%%%%%%%%%%%%%%%%%%%%%%%%%%%
\end{thm}
%%%%%%%%%%%%%%%%%%%%%%%%%%%%%%%%%%%%%%%
%\noindent 
%A few comments regarding this result are in order. 
A few comments:
%%%%%%%%%%%%%%%%%%%%%%%%%%%%%%%%%%%%%%%
\begin{itemize}
\item 
%%%%%%%%%%%%%%%%%%%%%%%%%%%%%%%%%%%%%%%
The identity $\boldsymbol{{\bf H Log}(X_3)}$ 
 is nothing else but the logarithm identity  \eqref{Eq:EFA-3-terms-Log}
and $\boldsymbol{{\bf H Log}(X_4)}$ coincides  with  the geometric identity 
$\boldsymbol{\big(\mathcal Ab_{X_4}\big)}$ hence 
is equivalent to Abel's relation $\boldsymbol{\big(\mathcal Ab\big)}$.  
In contrast, the four other identities $\boldsymbol{{\bf H Log}(X_r)}$ 
  for $r=5,6,7,8$ are new.  
  \sk
  \item 
%  \vspace{-0.3cm}
%%%%%%%%%%%%%%%%%%%%%%%%%%%%%%%%%%%%%%%
For any $i=1,\ldots,\kappa_r$,  the suitable branch of the hyperlogarithm  $AI_i^{r-2}$  from the statement of the theorem  is defined 
 in a precise and constructive way (see \ref{Eq:AI-r-2}). Furthermore, in Theorem \ref{Thm:main} we  prove an invariant algebraic version of Theorem \ref{Thm:MMain}  
by constructing (an algebraic equivalent of) each term $\epsilon_i \,AI_i^{r-2}(\phi_i)$  by means of the natural action of the  Weyl group $W_r$ on only one term, 
which we may assume to be $AI_1^{r-2}(\phi_1)$. 
%%%%%%%%%%%%%%%%%%%%%%%%%%%%%%%%%%%%%%%
\sk  \item 
%%%%%%%%%%%%%%%%%%%%%%%%%%%%%%%%%%%%%%%
At  least when $r\leq 7$, there is a conceptual interpretation of why $\boldsymbol{{\bf H Log}(X_r)}$ holds true in terms of the space $\mathbf C^{{\mathcal L}_r}$ freely spanned by the set ${\mathcal L}_r$ of lines contained in $X_r$. This space is acted upon in a natural way by  $W_r$ and from a representation-theoretic perspective, the left-hand side  of 
$\boldsymbol{{\bf H Log}(X_r)}$ can be interpreted as the image 
of the  signature representation ${\bf sign}_r$ of the Weyl group $W_r$ in the $(r-2)$-th wedge product of $\mathbf C^{{\mathcal L}_r}$.  The reason why $\sum_{i=1}^{\kappa_r} \epsilon_i \,AI_i^{r-2}\big(\phi_i\big)$ vanishes identically is that  ${\bf sign}_r$ does not appear with positive multiplicity in the decomposition 
of $\wedge^{r-2}\mathbf C^{{\mathcal L}_r}$ in irreducible $W_r$-modules. 
%%%%%%%%%%%%%%%%%%%%%%%%%%%%%%%%%%%%%%%
\end{itemize}
%%%%%%%%%%%%%%%%%%%%%%%%%%%%%%%%%%%%%%%

%%%%%%%%%%%%%%%%%%%%%%%%%%%%%%%%%%%%%%%
%%%%%%%%%%%%%%%%%%%%%%%%%%%%%%%%%%%%%%%
\subsection{Structure of the paper}
%%%%%%%%%%%%%%%%%%%%%%%%%%%%%%%%%%%%%%%
%%%%%%%%%%%%%%%%%%%%%%%%%%%%%%%%%%%%%%%
Throughout the paper, we work in the complex analytic or algebraic setting.

In Section \S\ref{S:Preliminary-Material}, we recall the basic facts about hyperlogarithms and del Pezzo surfaces which will be used in the rest of the paper. In particular, we explain how the functional identities satisfied by hyperlogarithms can be proved  algebraically ({\it cf.}\,Proposition \ref{Prop:Symbolic}).  Section \S\ref{S:Proof} is the main section and that is where Theorem \ref{Thm:MMain} is proved.  Using 
 Proposition \ref{Prop:Symbolic}, its proof is essentially 
 reduced to the verification that a certain (antisymmetric) tensorial identity 
${\bf hlog}=0$  
  holds true.  
We include at the end of  Section  \S\ref{S:Proof} some considerations 
  regarding possible generalizations to higher dimensions (blow-ups of projective spaces at general points).
 Finally in Section \S\ref{S:HLog3}, we make the identity $\boldsymbol{\bf H Log}(X_5)$ explicit in some affine coordinates.

%%%%%%%%%%%%%%%%%%%%%%%%%%%%%%%%%%%%%%%
%%%%%%%%%%%%%%%%%%%%%%%%%%%%%%%%%%%%%%%
\textcolor{black}{
\subsection{\bf Acknowledgements}
A.-M. Castravet was partially supported by the ANR grant FanoHK. Thanks go to Igor Dolgachev and Jenia Tevelev for several useful discussions. 
L. Pirio benefited from interesting early exchanges with Maria Chlouveraki and Nicolas Perrin, to whom he is grateful. He also thanks Thomas Dedieu and Vincent Guedj for their interest in this work.}

%%%%%%%%%%%%%%%%%%%%%%%%%%%%%%%%%%%%%%%
%%%%%%%%%%%%%%%%%%%%%%%%%%%%%%%%%%%%%%%
\section{\bf Preliminaries}
\label{S:Preliminary-Material}
%%%%%%%%%%%%%%%%%%%%%%%%%%%%%%%%%%%%%%%
%%%%%%%%%%%%%%%%%%%%%%%%%%%%%%%%%%%%%%%
In this section, we recall some properties of hyperlogarithms and del Pezzo surfaces. 

%%%%%%%%%%%%%%%%%%%%%%%%%%%%%%%%%%%%%%%
\subsection{Hyperlogarithms}
Hyperlogarithms are multivalued holomorphic functions on $\mathbf P^1$ which were used by Poincar\'e and Lappo-Danilevsky for building solutions to  linear differential equations  with regular singular points on 
the Riemann sphere.  As modern references about hyperlogarithms, the reader can consult \cite{W1,Brown} or \cite[\S2.3]{BPP}.

%%%%%%%%%%%%%%%%%%%%%%%%%%%%%%%%%%%%%%%
\subsubsection{}
\label{SS:again}
%%%%%%%%%%%%%%%%%%%%%%%%%%%%%%%%%%%%%%%
Let $n\geq1$ and $\sigma_1,\ldots, \sigma_m$  be $n$ pairwise distinct complex numbers. We set 
$$\Sigma=\{\sigma_1,\ldots, \sigma_m,\infty\}\subset \PP^1
\qquad \mbox{ and } \qquad Y=\PP^1\setminus \Sigma\, .
$$ 
The 1-forms $\omega_k=d\,{\rm Log}(z-\sigma_k)=dz/(z-\sigma_k)$ for $k=1,\ldots,m$ form a basis of the space 
$$\HH_{\Sigma}={\bf H}^0\Big(
\mathbf P^1, \Omega^1_{\mathbf P^1} \big( {\rm Log}\,\Sigma\big) 
\Big)$$
 of global rational 1-forms on $\PP^1$ with logarithmic poles along $\Sigma$. 
 
 We fix a base point $y\in Y$. 
For any word $\omega_{k_1}\omega_{k_2}\cdots \omega_{k_w}$ on the $\omega_k$'s, of length $w\geq 1$, we define the {\it hyperlogarithm associated to it at $y$} as the holomorphic germ at this point, denoted by  
%$L_{\underline{\omega}}^y$
$L_{ \omega_{k_1}\ldots \omega_{k_w} }^y$
defined inductively on the length $w$  by successive integrations performed on a sufficiently small neighborhood of $y$,  according to the following relations: 
%%%%%%%%%%%%%%%%%%%%%%%%%%%%%%%%%%%%%%%
$$
L^y_{\omega_{k_w}}(z) =\int^{\, z}_y \omega_{k_w}= {\rm Log}\left( 
\frac{z-\sigma_{\hspace{-0.02cm} k_w}}{y-\sigma_{\hspace{-0.02cm} k_w}}
 \right)
\qquad 
\mbox{ and } 
\qquad 
L_{ \omega_{k_1}\ldots \omega_{k_w} }^y (z) = 
\int^{z}_y \frac{ L_{ \omega_{k_2}\ldots \omega_{k_w} }^y (u)}{u-\sigma_{k_1}} du
\quad \mbox{ for } w>1\, , 
$$
%%%%%%%%%%%%%%%%%%%%%%%%%%%%%%%%%%%%%%%
 for any $z$ sufficiently close to $y$ on $\PP^1$.  The germ $L_{ \omega_{k_1}\ldots \omega_{k_w} }\in \mathcal O_{Y,y}$ admits analytic continuation along any continuous path 
$\gamma_z: [0,1]\rightarrow \PP^1\setminus \Sigma$ joining $y$ to an arbitrary point 
$z\in Y$. The value  at $z$ of this analytic continuation only depends on the homotopy class of $\gamma_z$ and is easily seen to coincide with the iterated integral of the tensor 
$\omega_{k_1}\otimes \cdots \otimes \omega_{k_w}
\in \big( \HH_{\Sigma}\big)^{\otimes w}$ along $\gamma_z$: one has
%%%%%%%%%%%%%%%%%%%%%%%%%%%%%%%%%%%%%%%
 \begin{equation*}
 %%%%%%%%%%%%%%%%%%%%%%%%%%%%%%%%%%%%%%%
 L^y_{ \omega_{k_1}\ldots \omega_{k_w} }(z) =\int_{\gamma_z}\frac{du}{u-\sigma_{k_1}}\otimes \frac{du}{u-\sigma_{k_2}}\otimes\ldots\otimes
\frac{du}{u-\sigma_{k_w}}\, .
%%%%%%%%%%%%%%%%%%%%%%%%%%%%%%%%%%%%%%%
\end{equation*}
 %%%%%%%%%%%%%%%%%%%%%%%%%%%%%%%%%%%%%%%
 
The germ  $L_{ \omega_{k_1}\ldots \omega_{k_w}}^y$ gives rise to a global but multivalued holomorphic  function on $Y$, with branch points at the $\sigma_k$'s, which we will still refer to as the hyperlogarithm associated to $\omega_{k_1}\ldots \omega_{k_w} $ 
and we denote by 
$L_{ \omega_{k_1}\ldots \omega_{k_w}}$. 

More formally, we  consider the  map (where ${\rm II}$ stands for ``Iterated Integral'') 
\begin{equation}\label{iterated integral}
{\rm II}^y_Y: \oplus_{w\geq0} \big( \HH_\Sigma)^{\otimes w}\longrightarrow \cO_{\PP^1,y}\, ,
\quad
\omega_{k_1}\otimes\omega_{k_2}\otimes\ldots\otimes\omega_{k_w}\longmapsto
%{\rm II}^y({\omega_1\omega_2\ldots \omega_m})=
L^y_{\omega_{k_1}\cdots \omega_{k_w}} \, .
\end{equation}
which in addition to being $\CC$-linear, can be proved to be a morphism of algebras if  $\oplus_{w\geq0} \big( \HH_\Sigma)^{\otimes w}$ is endowed with the so-called ``shuffle product'' (but we will not use this property in the rest of the paper). 
The image ${\rm Im}\big({\rm II}^y_Y\big)$ is a complex subalgebra of $ \cO_{\PP^1,y}$ and its elements are called 
 {\it (germs at $y$ of) hyperlogarithms}.  Moreover, the morphism 
 \eqref{iterated integral} is injective.  Consequently, for any germ of hyperlogarithm $L
 \in \cO_{\PP^1,y}$,  the minimum $w(L)$ of integers $w\geq 0$ such that $L$ belongs to 
 the image of 
 $\oplus_{w'\geq w } \big( \HH_\Sigma)^{\otimes w'}$ by 
${\rm II}^y_Y$ 
  is well-defined and is called the {\it weight} of $L$. 

As multivalued functions on $\mathbf P^1$, the monodromy of hyperlogarithms can be proved to be unipotent (see \cite[Thm.\,8.2]{W1}) from which it follows that these functions also form an algebra and that the notion of  weight still makes sense for them.

%%%%%%%%%%%%%%%%%%%%%%%%%%%%%%%%%%%%%%%
\subsubsection{}
%%%%%%%%%%%%%%%%%%%%%%%%%%%%%%%%%%%%%%%
The most classical example is for $m=2$ with $\sigma_1=0$ and $\sigma_2=1$ which encompasses the case of classical polylogarithms. Indeed, 
setting $\eta_0=dz/z$ and $\eta_1=dz/(1-z)$ in this special case, 
as multivalued hyperlogarithms on $\mathbf P^1\setminus \{0,1,\infty\}$, 
one has 
$$
{\rm Log}=L_{\eta_0}\, , \qquad R=\frac{1}{2}\left( L_{\eta_0\eta_1}- 
L_{\eta_1\eta_0}
\right)  \qquad \mbox{ and }\qquad {\bf L}{\rm i}_{n+1}=L_{\eta_0^{\otimes n}\eta_1}
\quad \mbox{ for any }\,  n\geq 0\, .$$

Working  locally with germs of hyperlogarithms is more involved but 
removes all ambiguity  regarding the choice of a branch  of the functions considered. 
For instance, for any $y\in \mathbf P^1\setminus \{0,1,\infty\}$, 
one has that the weight 2  hyperlogarithm  at $y$ whose symbol is $\frac12(\eta_0\eta_1-\eta_1\eta_0)$ is the holomorphic function defined by 
$$
R^y(z)=
\frac{1}{2}\left(L_{\eta_0\eta_1}(z)-L_{\eta_1\eta_0}(z)\right)=\frac{1}{2}\bigintsss_{y}^z\left(\frac{{\rm Log}\left(\frac{u-1}{y-1}\right)}{u} - \frac{{\rm Log}\left(\frac{u}{y}\right)}{u-1}\right)\, du$$ 
for any $z\in (\mathbf P^1,y)$.   This hyperlogarithm has to be seen as a holomorphic version, localized at $y$, of Rogers' dilogarithm defined in \ref{Eq:R}.

%%%%%%%%%%%%%%%%%%%%%%%%%%%%%%%%%%%%%%%
\subsubsection{}
%%%%%%%%%%%%%%%%%%%%%%%%%%%%%%%%%%%%%%% 
We now define the hyperlogarithms involved in this paper, noted by $AI^{w}$.  Even if our results is for weights $w$ less than or equal to 6, the definition of $AI^{w}$ is completly uniform in $w$,  hence we will not impose any restriction on the weight in this subsection. 
\sk

We use the notation of \S\ref{SS:again} again: $\Sigma=\{\sigma_1,\ldots,\sigma_m,\infty\}$, $\omega_k=dz/(z-\sigma_k)$ for $k=1,\ldots,m$, etc.
We introduce a special class of hyperlogarithms  on $\PP^1$, with respect to $\Sigma$, of weight $m=|\Sigma|-1$.  
For a $\CC$-vector space $V$, we identify $\wedge^m V$ with its image in $V^{\otimes m}$ under the standard embedding: 
$$\wedge^m V\hra V^{\otimes m},\qquad v_1\wedge\ldots\wedge v_m\longmapsto
\frac{1}{m!}\bigg(\sum_{\tau \in \mathfrak S_{m}}(-1)^\tau v_{\tau(1)}\otimes\ldots\otimes v_{\tau(m)}\bigg)\, ,  $$
where $(-1)^\tau$ stands for the signature of $\tau$ for any permutation $\tau \in \mathfrak S_{m}$. 

%%%%%%%%%%%%%%%%%%%%%%%%%%%%%%%%%%%%%%%%
\begin{defn}
\label{AI}
%%%%%%%%%%%%%%%%%%%%%%%%%%%%%%%%%%%%%%%%
The (complete) {\bf anti-symmetric hyperlogarithm} $AI^m$ of weight $m$ on $\PP^1$, 
with respect to $\Sigma$, is the hyperlogarithm whose germ at any $y\in Y=\PP^1\setminus \Sigma$ is 
obtained by taking the image of 
$\omega_1\wedge\ldots\wedge\omega_m\in \wedge^m\HH_\Sigma\subset \big(\HH_\Sigma\big)^{\otimes m} $  under the map \eqref{iterated integral}: as germs at $y$, one has 
%%%%%%%%%%%%%%%%%%%%%%%%%%%%%%%%%%%%%%%%
$$AI_{\Sigma}^{m}={\rm II}^y_Y\big(\omega_1\wedge\ldots\wedge\omega_m\big)\, .$$
\end{defn}
%%%%%%%%%%%%%%%%%%%%%%%%%%%%%%%%%%%%%%%%

One verifies that $\omega_1\wedge\ldots\wedge\omega_m\in \wedge^m \HH_\Sigma$  is canonically defined, up to a sign. It follows that $\pm AI_{\Sigma}^{m}$ is canonically defined by $\Sigma$. Here are some easy remarks about the first three examples: 
%%%%%%%%%%%%%%%%%%%%%%%%%%%%%%%%%%%%%%%%
\begin{enumerate}
%%%%%%%%%%%%%%%%%%%%%%%%%%%%%%%%%%%%%%%%
\item[$-$] $m=1$  and $\sigma_1=0$;  one has   $AI^1_{\{0,\infty\}}={\rm Log}$ up to sign; 
\sk 
%%%%%%%%%%%%%%%%%%%%%%%%%%%%%%%%%%%%%%%%
\item[$-$] $m=2$  and $\sigma_1=0$, $\sigma^2=1$;  up to sign, one recovers 
the holomorphic version of 
Rogers' dilogarithm discussed above
since $AI^2_{\{0,1,\infty\}}=R^y$ as germs at any $y \in \PP^1\setminus \{0,1,\infty\}$; 
\sk 
%%%%%%%%%%%%%%%%%%%%%%%%%%%%%%%%%%%%%%%%
\item[$-$] the case $m=3$  is new since,   the weight 3 antisymmetric hyperlogarithm has not been considered in the literature before as far we know. For any $y\in \mathbf P^1\setminus \Sigma$ with $\Sigma=\{\sigma_1,\sigma_2,\sigma_3,\infty\}$, one can give an explicit integral expression for  
$AI_\Sigma^3$ (see \eqref{Eq:AI-a,b,c}).  However, one can prove that 
$AI_{\Sigma}^{3}$ can be expressed as the following linear combination of products of 
antisymmetric polylogarithms of weight 1 or 2 since 
for suitable choices of the sign of $AI_\Sigma^3$ and of the weight two hyperlogarithms $AI_{\Sigma\setminus 
\{\sigma_k\}} ^2$'s for $k=1,2,3$, the following relation holds true
$$AI_{\Sigma}^3\big(z\big)=
 \frac13 %\bigg( 
 \sum_{k=1}^3 (-1)^{k-1} {\rm Log}\left( 
 \frac{z-\sigma_k}{y-\sigma_k}
 \right)\cdot 
  AI^2_{\Sigma\setminus 
\{\sigma_k\}} \big(z\big)
 $$ 
 for any $z\in \mathbf P^1$ sufficiently close of the previously fixed base point  $y$. 
%%%%%%%%%%%%%%%%%%%%%%%%%%%%%%%%%%%%%%%%
\end{enumerate}
%%%%%%%%%%%%%%%%%%%%%%%%%%%%%%%%%%%%%%%%

%%%%%%%%%%%%%%%%%%%%%%%%%%%%%%%%%%%%%%%
%%%%%%%%%%%%%%%%%%%%%%%%%%%%%%%%%%%%%%%
\subsubsection{\bf Pull-backs of hyperlogarithms}
\label{SS: General case}
%%%%%%%%%%%%%%%%%%%%%%%%%%%%%%%%%%%%%%%
%%%%%%%%%%%%%%%%%%%%%%%%%%%%%%%%%%%%%%%
Let $Y$ be a (not necessarily compact) complex manifold and let $\HH\subset {\bf H}^0\big(Y,\Omega^1_Y)$ be a subspace of 
holomorphic $1$-forms on $Y$, such that: 
%%%%%%%%%%%%%%%%%%%%%%%%%%%%%%%%%%%%%%%
\begin{equation}
\label{*}
%%%%%%%%%%%%%%%%%%%%%%%%%%%%%%%%%%%%%%%
\text{\it For all } \, \omega,\omega'\in \HH, \text{\it one has } \, 
d\omega=0 \ \, \text{\it  and } \, \  \omega\wedge\omega'=0\,.
%%%%%%%%%%%%%%%%%%%%%%%%%%%%%%%%%%%%%%%
\end{equation}
%%%%%%%%%%%%%%%%%%%%%%%%%%%%%%%%%%%%%%%
The conditions \eqref{*} are satisfied if, for example,  
$\HH=\phi^* {\bf H}^0\big(C,\Omega^1_C)$, for some regular submersion $\phi: Y\rightarrow C$, with $C$ a smooth (not necessarily compact) curve. 
The conditions  \eqref{*}  ensure that for any holomorphic $1$-forms $\omega_1\ldots,\omega_w$ in $\HH$, the iterated integral 
${\rm II}_{\omega_1\omega_2\ldots \omega_w}^y
=\int^\bullet \omega_1\otimes \omega_2\otimes \cdots  \otimes \omega_w  $, defined inductively as in \S\ref{SS:again}, depends only on the homotopy class of the path $\gamma_z$. Hence, for all $m\geq1$, 
there is $\CC$-linear map ${\rm II}^y_Y: \oplus_{w\geq0}\HH^{\otimes w}\ra\cO_{Y,y}$, defined as in (\ref{iterated integral}).  Furthermore, this map is an injective 
morphism of complex algebras.\sk 

A special situation occurs when the conditions \eqref{*}  are not necessarily satisfied for all elements of $\HH$, but there exist subspaces 
$\HH_i\subset \HH$ for $i=1\ldots, d$, such that for each $i$, $\HH_i$ satisfies \eqref{*}. In this case, we have again a well-defined \emph{injective} 
$\CC$-linear map given by the iterated integrals on the subspaces $\sum_{i=1}^d\big( \HH_i\big)^{\otimes w}\subset\HH^{\otimes w}$, for all $w\geq0$:
%%%%%%%%%%%%%%%%%%%%%%%%%%%%%%%%%%%%%%%
\begin{equation}\label{several H's}
{\rm II}^y_Y: \oplus_{w\geq0}\bigg(\sum_{i=1}^d\HH_i^{\otimes w}\bigg)\ra\cO_{Y,y}. 
\end{equation}
%%%%%%%%%%%%%%%%%%%%%%%%%%%%%%%%%%%%%%%

%%%%%%%%%%%%%%%%%%%%%%%%%%%%%%%%%%%%%%%
%%%%%%%%%%%%%%%%%%%%%%%%%%%%%%%%%%%%%%%
\subsubsection{\bf Hyperlogaritms for webs}
\label{SS: webs}
%%%%%%%%%%%%%%%%%%%%%%%%%%%%%%%%%%%%%%%
Fix $m\geq1$. The situation we consider here is when  $X$ a  complex projective manifold, $\phi_i: X\rightarrow\PP^1$ surjective morphisms  (with $i=1,\ldots, d$), such that 
there exists subsets of $m+1$ distinct points  $\Sigma_i=\{\sigma_i^1,\ldots, \sigma_i^m,\infty\}\subset\PP^1$, such that 
$\phi_i: X\setminus \phi_i^{-1}(\Sigma_i)\ra\PP^1\setminus\Sigma_i$ is a regular submersion for all $i$, and the union $D\subset X$ of all divisors 
$\phi_i^{-1}(\sigma^k_i)\subset X$, for all $i$ and $k=1,\ldots,m$  is such that  one has $d\phi_i\wedge d\phi_j\neq 0$ on $Y=X\setminus D$, for all $i,j=1,\ldots, d$ distinct. 
Then the  maps $\phi_i$ ($i=1,\ldots,d$) define  a regular \emph{$d$-web of hypersurfaces} on $Y$.  

 In such a situation, we consider the following notations: 
%%%%%%%%%%%%%%%%%%%%%%%%%%%%%%%%%%%%%%%
\begin{itemize}
%%%%%%%%%%%%%%%%%%%%%%%%%%%%%%%%%%%%%%%
\item  we denote by $\HH={\bf H}^0\big(X,\Omega_X^1({\rm Log} D)\big)\subset {\bf H}^0\big(Y,\Omega^1_Y\big)$ the space of logarithmic 1-forms on $X$ with logarithmic 
poles along to $D$; 
\sk 
%%%%%%%%%%%%%%%%%%%%%%%%%%%%%%%%%%%%%%%
\item for $i=1,\ldots,d$, we set $Y_i=\mathbf P^1\setminus \Sigma_i$ and 
\sk
\begin{itemize}
\sk 
%%%%%%%%%%%%%%%%%%%%%%%%%%%%%%%%%%%%%%%
%\item[$-$] $Y_i=\mathbf P^1\setminus \Sigma_i$; 
%\sk 
%%%%%%%%%%%%%%%%%%%%%%%%%%%%%%%%%%%%%%%
\item[$-$] $\theta^j_i={dz}/{(z-\sigma_i^k)}$
for $k=1,\ldots,m$, which form a basis of
$$\HH_{\Sigma_i}={\bf H}^0\Big(\PP^1,\Omega^1_{\PP^1}\big({\rm Log}\,\Sigma_i\big)\Big)$$
%%%%%%%%%%%%%%%%%%%%%%%%%%%%%%%%%%%%%%%
\item[$-$] 
$\Theta^j_i=\phi_i^*\big( \theta^j_i\big)=
{d\phi_i}/{(\phi_i-\sigma_i^k)}$
for $k=1,\ldots,m$, which form a basis of 
$$
\HH_i=\phi_j^*\Big(\HH_{\Sigma_i}\Big) \subset\HH; $$
%%%%%%%%%%%%%%%%%%%%%%%%%%%%%%%%%%%%%%%
\item[$-$]  
$\theta_i=\theta^1_i\wedge\ldots\wedge\theta^m_i \in\wedge^m\HH_{\Sigma_i}\subset \big(\HH_{\Sigma_i}\big) ^{\otimes m}$ and 
$\Theta_i=\phi_i^*\big(\theta_i\big)=\wedge_{k=1}^m 
\Theta^k_i \in\wedge^m\HH_{i}\subset \big(\HH_{i}\big) ^{\otimes m}$;
\sk 
%%%%%%%%%%%%%%%%%%%%%%%%%%%%%%%%%%%%%%%
%%%%%%%%%%%%%%%%%%%%%%%%%%%%%%%%%%%%%%%
\end{itemize}
%%%%%%%%%%%%%%%%%%%%%%%%%%%%%%%%%%%%%%%
\item for any $y\in Y$, we set $y_i=\phi_i(y)\in \PP^1\setminus \Sigma_i$  and we consider the (germs of) weight $m$ hyperlogarithms
$$
AI^m_{\Sigma_i}={\rm II}_{Y_i}^{y_i}\big( \theta_i \big) \in \mathcal O_{\PP^1,y_i}
\quad \mbox{for } i=1,\ldots,d\, 
\qquad \mbox{and} \qquad 
AI^m_{i}={\rm II}_Y^y\big( \Theta_i \big) \in \mathcal O_{Y,y}
\, .
$$
%%%%%%%%%%%%%%%%%%%%%%%%%%%%%%%%%%%%%%%
\end{itemize}
%%%%%%%%%%%%%%%%%%%%%%%%%%%%%%%%%%%%%%%

One verifies easily that for any $i$,  the following relation holds true as germs on $Y$ at $y$: 
$$AI^m_{i}= AI^m_{\Sigma_i} \circ \phi_i  .
$$

For any $i=1,\ldots,d$, the hyperlogarithm $AI^m_{i}$ (or equivalently $AI^m_{\Sigma_i} $) is only well-defined up to multiplication by $-1$. For each $i$, we fix one of the two possible choices for $AI^m_{i}$. 
The following result, although elementary to prove, is key since it will allow  us to handle algebraically the functional identity we want to establish in \S3:% (we use the notations above):  

%%%%%%%%%%%%%%%%%%%%%%%%%%%%%%%%%%%%%%%
\begin{prop} 
\label{Prop:Symbolic}
%%%%%%%%%%%%%%%%%%%%%%%%%%%%%%%%%%%%%%%
%We use the notation above.  
For $c_1\ldots, c_d\in \CC$, the following statements are equivalent: 
\begin{enumerate}
\item[i.]  One has $\sum_{i=1}^d c_i\,\Theta_i=0$ in $\wedge^n\HH\subset \HH^{\otimes n}$. 
\sk
%%%%%%%%%%%%%%%%%%%%%%%%%%%%%%%%%%%%%%%
\item[ii.] There exists $y\in Y$ such that $\sum_{i=1}^d c_i\,{AI}^m_{\Sigma_i}(\phi_i)=0$ as a holomorphic germ at $y$ on $Y$.
\sk
%%%%%%%%%%%%%%%%%%%%%%%%%%%%%%%%%%%%%%%
\item[iii.] For any $y\in Y$, $\sum_{i=1}^d c_i\,{AI}^m_{\Sigma_i}(\phi_i)=0$ as a holomorphic germ at $y$ on $Y$.
\sk
%%%%%%%%%%%%%%%%%%%%%%%%%%%%%%%%%%%%%%%
\item[iv.] One has $\sum_{i=1}^d c_i\,{AI}^m_{\Sigma_i}(\phi_i)=0$ as multivalued functions on $Y$. 
\end{enumerate}

%%%%%%%%%%%%%%%%%%%%%%%%%%%%%%%%%%%%%%%
\end{prop}
%%%%%%%%%%%%%%%%%%%%%%%%%%%%%%%%%%%%%%%
\begin{proof}
For a point $y\in Y$, we have $\sum_{i=1}^d c_i\,{AI}^m_{\Sigma_i}(\phi_i)={\rm II}^y_Y(\sum_{i=1}^d c_i\Theta_i)$, where ${\rm II}^y_Y$ is the integration map in (\ref{several H's}). The statement now follows from the fact that this map is injective. 
\end{proof}

%%%%%%%%%%%%%%%%%%%%%%%%%%%%%%%%%%%%%%%
%%%%%%%%%%%%%%%%%%%%%%%%%%%%%%%%%%%%%%%
\subsection{Del Pezzo surfaces}
\label{SS:dPd}
%%%%%%%%%%%%%%%%%%%%%%%%%%%%%%%%%%%%%%%
%%%%%%%%%%%%%%%%%%%%%%%%%%%%%%%%%%%%%%%
Del Pezzo surfaces are smooth projective surfaces with ample anti-canonical line bundle. A del Pezzo surface is isomorphic to either $\PP^1\times\PP^1$ or 
a blow-up ${\bf Bl}_{p_1,\ldots,p_r}\big( \PP^2\big)$ ($r\geq8$) at $r$ points $p_1,\ldots,p_r$ 
in general position in $\PP^2$. 

In what follows we consider del Pezzo surfaces 
$X_r={\bf Bl}_{p_1,\ldots,p_r}\big( \PP^2\big)$ for $3\leq r\leq 8$. We fix a blow-up map $\beta=\beta_r: X_r\ra\PP^2$. 
We refer to \cite[Chap.\;IV]{Manin} or \cite[Chap.\;8]{Dolgachev}
\black
for general facts about del Pezzo surfaces. Here we make a list of the properties that we will use. 
\black
%%%%%%%%%%%%%%%%%%%%%%%%%%%%%%%%%%%%%%%
%%%%%%%%%%%%%%%%%%%%%%%%%%%%%%%%%%%%%%%
\begin{enumerate}
%%%%%%%%%%%%%%%%%%%%%%%%%%%%%%%%%%%%%%%
%%%%%%%%%%%%%%%%%%%%%%%%%%%%%%%%%%%%%%%
\item[(1).]
The Picard group $\Pic(X_r)$ is a free abelian group generated by the classes $e_i$ of the exceptional divisors $\beta^{-1}(p_i)$ (for $i=1,\ldots,r$) and the class $h$ of the preimage under $\beta$ of a general line in $\PP^2$.  The intersection pairing on $X_r$ is determined by $h^2=1$, $h\cdot e_i=0$, $e_i\cdot e_j=-\delta_{ij}$, for all $i,j\in\{1,\ldots,r\}$. 
\mk
%%%%%%%%%%%%%%%%%%%%%%%%%
\item[(2).]
The canonical divisor is $K=K_{X_r}=-3h+\sum_{i=1}^r e_i$ and the \emph{degree} of $X_r$ is $(-K)^2=9-r$. 
\mk
%%%%%%%%%%%%%%%%%%%%%%%%%
%\vspace{-0.35cm}
\item[(3).] A \emph{line} on $X_r$ is a smooth curve $\ell\subset X_r$ with $ K \cdot \ell=\ell^2=-1$. Such a line is necessarily a smooth rational curve and it can be naturally identified with its class in $\Pic(X_r)$.
We denote by $\cL_r$ the set of lines on $X_r$. 
\mk
%%%%%%%%%%%%%%%%%%%%%%%%%
\item[(4).] A \emph{conic} on $X_r$ is a  curve $C\subset X_r$ with $C\cdot K=-2$ and  $C^2=0$.  When $C$ is smooth, it is necessarily a smooth rational curve. Otherwise, it is the sum of
two concurrent lines on $X_r$. We denote by $\cK_r$ the set of conic classes. 

A \emph{conic fibration} $X_r\ra\PP^1$ is given by the complete linear system of a conic on $X_r$. Hence,  $\cK_r$ corresponds to the set of conic fibrations up to projective equivalence. 
\mk
%%%%%%%%%%%%%%%%%%%%%%%%%
\item[(5).] The orthogonal complement 
$K^\perp=\big\{ \, \alpha \in \Pic(X_r)\, \big\lvert \, \alpha\cdot K=0\, \big\}$ is free of rank $r$ and spanned by the classes 
%%%%%%%%%%%%%%%%%%%%%%%%%
$$\alpha_i=e_i-e_{i+1}\quad \mbox{for } i=1,\ldots,r-1
\qquad \mbox{ and } \qquad 
\alpha_r=3h-e_1-e_2-e_3\,.$$
%%%%%%%%%%%%%%%%%%%%%%%%%
Together with the positive definite symmetric form $ - ( \cdot,\cdot)\lvert_{K^\perp}$ coming from the intersection pairing,  
the $\alpha_i$'s define a root system of type $E_r$, with the convention that $E_4=A_4$, $E_5=D_5$, see the following figure: 
% ({\it cf.}\,Figure \ref{Fig:DynkinDiagram}). 
 %%%%%%%%%%%%%%%%%%%%%%%%%%%%%%%
%\vspace{-0.35cm}
 \begin{figure}[h!]
\begin{center}
\scalebox{2}{
 \includegraphics{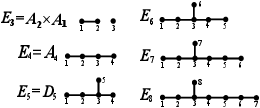}}
 \vspace{-0.1cm}
\caption{Dynkin diagrams $E_r$ (with $i$ standing for %the fundamental root 
 $\alpha_i$ for any $i=1,\ldots,r$)} 
\label{Fig:DynkinDiagram}
\end{center}
\end{figure}
 %%%%%%%%%%%%%%%%%%%%%%%%%%%%%%%
\mk
%%%%%%%%%%%%%%%%%%%%%%%%%
\vspace{-0.35cm}
\item[(6).]  For any $i=1,\ldots,r$, the map 
%%%%%%%%%%%%%%%%%%%%%%%%%
\begin{equation}
\label{Eq:s-alpha-i}
%%%%%%%%%%%%%%%%%%%%%%%%%
s_{\alpha_i} : % \Pic(X_r)\longrightarrow  \Pic(X_r) , \, 
\beta\longmapsto \beta+\big(\,\beta\cdot\alpha_i\,\big)\,\alpha_i
%%%%%%%%%%%%%%%%%%%%%%%%%
\end{equation}
%%%%%%%%%%%%%%%%%%%%%%%%%
 is an involutive automorphism of 
$\big( \Pic(X_r),  ( \cdot,\cdot)\big)$ which lets $K$ invariant.  The restrictions of 
the $s_{\alpha_i}$'s  to $R_r=K^\perp\otimes_{\mathbf Z} \mathbf R$ are orthogonal reflections and they 
generate a Weyl group of type $E_r$, denoted by $W_r$. 
In particular, $W_r$ is finite.
\mk
%%%%%%%%%%%%%%%%%%%%%%%%%
\item[(7).]
For simplicity, we set  $s_i=s_{\alpha_i}$ for any $i$.
When $i=1,\ldots, r-1$, the reflection $s_i$ acts on $\Pic(X_r)$ by interchanging $e_i$ with $e_{i+1}$, 
leaving other exceptional classes $e_k$ and $h$ fixed. The reflection $s_r$ acts as a Cremona transformation, i.e., one has  $s_r(h)=2h-e_1-e_2-e_3$ and $s_r(e_i)= h-e_j-e_k$ 
for $\{i,j,k\}=\{1,2,3\}$ and $s_r$ leaves $e_k$ fixed for $k=4,\ldots,r$. 
\mk
%%%%%%%%%%%%%%%%%%%%%%%%%
\item[(8).]  For an element $w\in W_r$, we denote $(-1)^w=(-1)^{l(w)}\in \{\, \pm 1\,\}$ the \emph{signature} of $w$. Here $l(w)$ stands for the length of $w$
which by definition is  the smallest non negative integer $m$ 
such that one can write $w=s_{i_1}\cdots s_{i_m}$ for some $i_1,\ldots,i_m$ in $\{1,\ldots,r\}$. The map $W_r\rightarrow \{\, \pm 1\,\}$, $w\mapsto (-1)^w$ is a group morphism, called the {\it signature}. The associated   {\it signature representation} is the unique non trivial representation of $W_r$ of dimension 1.
\mk
%%%%%%%%%%%%%%%%%%%%%%%%%
\item[(9).] Any line $\ell$ (respectively, any conic class $\mc$) on $X_r$ belongs to the $W_r$-orbit of the exceptional divisor $e_1$ (respectively,  $h-e_1$). 
This follows from Noether's inequality (e.g., see \cite[p.\,288]{Dolgachev2}). 
Equivalently:  $W_r$ acts transitively on the set $\cL_r$ of lines (respectively, on the set $\cK_r$ of conic classes). 
\mk
%%%%%%%%%%%%%%%%%%%%%%%%%
\item[(10).] Any conic fibration $\phi_\mc: X_r\ra \PP^1$ corresponding to a conic class $\mc$ has exactly $r-1$ reducible fibers, each a union of two lines intersecting at a point. 
In particular, each conic class is of the form $\mc=\ell+\ell'$, with $\ell$, $\ell'$ lines such that $\ell\cdot \ell'=1$. We will often 
write $\ell+\ell'$ to indicate the reducible conic $\ell\cup\ell'$. 
\mk
%%%%%%%%%%%%%%%%%%%%%%%%%
\item[(11).] 
%The Weyl group $W_r$ naturally acts by permutations on the set of lines $\cL_r$.  
For $r>3$, the stabilizer $W_{e_r}$ of $e_r\in \cL_r$  is generated by the reflections $s_i$'s for $i$ ranging from 1 to $r$  and distinct from $r-1$. It follows that $W_{e_r}$ is isomorphic to the Weyl group associated to the Dynkin diagram $E'_{r-1}$ obtained by removing  the $(r-1)$-th 
 node as well as the edge adjacent to it from $E_r$, that is $W_{e_r}\simeq W(E_{r-1})$.  
In particular, for $r>3$ we have 
$l_r=\lvert  \cL_r\lvert = \lvert W_r\lvert / \lvert W_{r-1}\lvert$.   For $r=3$, one has 
 $W_{e_3}=\langle s_1\rangle \simeq \{\pm 1\}$ and $l_3=\lvert \mathcal L_3\lvert 
 %=\lvert W_3 \lvert / 2
 =6 $.

\mk
%%%%%%%%%%%%%%%%%%%%%%%%%
\item[(12).] The stabilizer $W_{\mc_1}$ of $\mc_1=h-e_1\in \cK$  is generated by the reflections $s_2,\ldots, s_r$. 
This subgroup of $W_r$ is isomorphic to the Weyl group associated to the Dynkin diagram $E''_{r-1}$ obtained  by removing the first node  as well as the edge adjacent to it from $E_r$, which hence is of type $D_{r-1}$.\footnote{ Here we use the convention that $D_2=A_1\times A_1$ and $D_3=A_3$.} In particular, we have 
$\kappa_r=\lvert  \cK_r\lvert = \lvert W_r\lvert / \lvert W(D_{r-1})\lvert
= \lvert W_r\lvert /(2^{r-2} (r-1)!)$. 

\mk
%%%%%%%%%%%%%%%%%%%%%%%%%
\item[(13).] 
For any mutually disjoint $r-2$ lines $\ell_1,\ldots,\ell_{r-2}$, there exists an element $w$ of the Weyl group $W$ such that 
$w\cdot e_i=\ell_i$ for all $i=1,\ldots, r-2$ ({\it cf.}\,Corollary 26.8.(i) in \cite{Manin}). 
\end{enumerate}

Some numerical invariants associated to the Weyl groups $W_r$ and the sets 
of lines and conics $\cL_r$ and $\cK_r$ are gathered in the following table: 
%%%%%%%%%%%%%%%%%%%%%%%%%%%%%%
\begin{table}[!h]
\scalebox{1}{\begin{tabular}{|l||c|c|c|c|c|c|}
\hline
${}^{}$ \hspace{0.4cm}  
\begin{tabular}{c}\vspace{-0.35cm}\\
$\boldsymbol{r}$
\vspace{0.13cm}
\end{tabular} & $\boldsymbol{3}$ &  $\boldsymbol{4}$ & $\boldsymbol{5}$  & $\boldsymbol{6}$ & $\boldsymbol{7}$  & $\boldsymbol{8}$\\ \hline \hline
%%%%%%%%%%%%%%%%%%%%%%%%%%%%%%
${}^{}$ \quad \begin{tabular}{c}\vspace{-0.35cm}\\
$\boldsymbol{E_r}$
\vspace{0.1cm}
\end{tabular}
 & $A_2\times A_1$ & $A_4$ & $D_5$   & $E_6$ & $E_7$  & $E_8$  \\ \hline 
${}^{}$ \quad 
\begin{tabular}{c}\vspace{-0.35cm}\\
$\boldsymbol{W_r=W(E_{r})}$
\vspace{0.1cm}
\end{tabular}
& 
$\mathfrak S_3\times \mathfrak S_2$
 & $\mathfrak S_5$ & $\big(\mathbf Z/2\mathbf Z)^4
 \ltimes  \mathfrak S_5  
$   & $W(E_6)$ & $W(E_7)$ & $W(E_8)$  \\ \hline 
%%%%%%%%%%%%%%%%%%%%%%%%%%%%%%
${}^{}$ \quad  
\begin{tabular}{c}\vspace{-0.35cm}\\
$\boldsymbol{\omega_r=\lvert W_r\lvert}$
\vspace{0.1cm}
\end{tabular}
& $12$ & $5!$ & $2^4\cdot
5!
$   & $2^7\cdot 3^4\cdot 5$ & $2^{10}\cdot 3^4 \cdot 5\cdot7$ & 
$2^{14}\cdot 3^5 \cdot 5^2\cdot7$
  \\ \hline 
%%%%%%%%%%%%%%%%%%%%%%%%%%%%%%
${}^{}$ \quad 
\begin{tabular}{c}\vspace{-0.35cm}\\
$\boldsymbol{l_r={\lvert \mathcal L_r \lvert}}$
%$\boldsymbol{l_r={\lvert \mathcal L_r \lvert}={\lvert W_r \lvert}\, \big/\, {\big\lvert  W_{r-1} \big\lvert}}$ ($r>3$)
\vspace{0.1cm}
\end{tabular}
&6 & 10 &  16  & 27 & 56 & 240  \\ \hline
%%%%%%%%%%%%%%%%%%%%%%%%%%%%%%
${}^{}$ \quad 
\begin{tabular}{c}\vspace{-0.35cm}\\
$\boldsymbol{\kappa_r={\lvert \mathcal K_r \lvert}}$
%$\boldsymbol{\kappa_r={\lvert \mathcal K_r \lvert}={\lvert W_r\lvert}\, \big/\, {\big\lvert W''_{r-1} \big\lvert} }$
\vspace{0.13cm}
\end{tabular}
& 3 & 5 &    10 & 27 & 126 & 2160 \\
\hline
%%%%%%%%%%%%%%%%%%%%%%%%%%%%%%
\end{tabular}}
\bk 
\caption{}
\end{table}

An  important ingredient in our approach is that  for each case $r\in \{3,\ldots,8\}$,  there are explicit descriptions of both sets $\cL_r$ and $\cK_r$ (when those are seen as subsets of $\Pic (X_r)$). We
mention only the case when $r = 8$ (from which the other cases can be easily deduced) and refer to \cite[\S26]{Manin} and \cite{Dolgachev} for details and proofs. 

When viewed as elements of ${\rm Pic}(X_r)$, any line or conic class is uniquely determined by the tuple of its integer coordinates  $(d,m_1,\ldots,m_r)\in \mathbf Z^{r+1}$ with respect to the basis $(h,-e_1,\ldots,-e_r)$ of the Picard lattice.  Let the 
{\it type} of a coordinate $(r+1)$-tuple $(d,m_1,\ldots,m_r)$ by a symbol $\big(d\,;\,  k_1^{n_1},\ldots, k_s^{n_s}\big)$  for some  integers $k_t\neq 0$ and $n_t>0$ for $t=1,\ldots,s\leq r$, with the defining property that among the non zero $m_1,\ldots,m_r$, exactly $n_t$ are equal to $k_t$, this for all $t$ ranging from 1 to $s$ (for example, the type of $(6,2,2,2,3,2,2,2,0)\in \mathbf Z^{9}$ is $\big(\,6\,; \,3,\ 2^6\, \big)$, etc).  

In the table below, we list all the types of lines and conics classes on $X_8$, and indicate the number of  classes there are for each type (see \cite[Prop.\,26.1]{Manin} and \cite[\S8.8]{Dolgachev}).

\begin{table}[ht]
\begin{center}
%%%%%%%%%%%%%%%%%%%%%%%%%%%%
\begin{tabular}{c} 
\vspace{-5.4cm}
\\
\begin{tabular}{|l|c|}
  \hline
    \multicolumn{2}{|c|}{
  \begin{tabular}{c}  
  \vspace{-0.35cm}\\
  {\bf Lines on $\boldsymbol{X_8}$}\vspace{0.06cm}
   \end{tabular} }
    \\
      \hline
      \begin{tabular}{c}  
  \vspace{-0.35cm}\\
  ${}^{}$ \hspace{0.07cm}   {\bf Types} 
  \vspace{0.06cm}
   \end{tabular} 
   & {\bf Number of such}\\
    \hline  \hline 
  \begin{tabular}{c}  
  \vspace{-0.36cm}\\
   ${}^{}$ \hspace{-0.15cm}   $\big(\,0\,;\, -1\,\big)$\vspace{0.04cm}
   \end{tabular} 
      & 8 \\
    \hline
     %%%%%%%%%%%%%%%
       \begin{tabular}{c}  
  \vspace{-0.36cm}\\
${}^{}$ \hspace{-0.15cm}  $\big(\,1\,;\, 1^2\,\big)$   \vspace{0.04cm}
   \end{tabular}     & 28 \\
    \hline
    %%%%%%%%%%%%%%%
           \begin{tabular}{c}  
  \vspace{-0.36cm}\\
${}^{}$ \hspace{-0.15cm}  $\big(\,2\,;\, 1^5\,\big)$   \vspace{0.04cm}
   \end{tabular}     & 56 \\
    \hline
    %%%%%%%%%%%%%%%
           \begin{tabular}{c}  
  \vspace{-0.36cm}\\
${}^{}$ \hspace{-0.15cm}  $\big(\,3\,;\, 2, 1^6\,\big)$   \vspace{0.04cm}
   \end{tabular}     &  56 \\
    \hline
    %%%%%%%%%%%%%%%
       \begin{tabular}{c}  
  \vspace{-0.36cm}\\
${}^{}$ \hspace{-0.15cm}  $\big(\,4\,;\, 2^3, 1^5\,\big)$   \vspace{0.04cm}
   \end{tabular}     & 56 \\
    \hline
    %%%%%%%%%%%%%%%
       \begin{tabular}{c}  
  \vspace{-0.36cm}\\
${}^{}$ \hspace{-0.15cm}  $\big(\,5\,;\, 2^6, 1^2\,\big)$   \vspace{0.04cm}
   \end{tabular}     & 28 \\
    \hline
    %%%%%%%%%%%%%%%
       \begin{tabular}{c}  
  \vspace{-0.36cm}\\
${}^{}$ \hspace{-0.15cm}  $\big(\,6\,;\, 3, 2^7\,\big)$   \vspace{0.04cm}
   \end{tabular}     &  8 \\
    \hline
    %%%%%%%%%%%%%%%
\end{tabular}
\end{tabular}
%%%%%%%%%%%%%%%%%%%%%%%%%%%%
\quad 
%%%%%%%%%%%%%%%%%%%%%%%%%%%%
\begin{tabular}{|l|c|}
  \hline
    \multicolumn{2}{|c|}{
  \begin{tabular}{c}  
  \vspace{-0.35cm}\\
  {\bf Conic classes on $\boldsymbol{X_8}$}\vspace{0.06cm}
   \end{tabular} }
    \\
      \hline
      \begin{tabular}{c}  
  \vspace{-0.35cm}\\
  ${}^{}$ \hspace{0.07cm}   {\bf Types} 
  \vspace{0.06cm}
   \end{tabular} 
   & {\bf Number of such}\\
    \hline  \hline 
       %%%%%%%%%%%%%%%
%  \begin{tabular}{c}  
%  \vspace{-0.36cm}\\
%   ${}^{}$ \hspace{0.25cm}   $-$\vspace{0.04cm}
%   \end{tabular} 
%      & 0 \\
%    \hline
     %%%%%%%%%%%%%%%
       \begin{tabular}{c}  
  \vspace{-0.36cm}\\
${}^{}$ \hspace{-0.15cm}  $\big(\,1\,;\, 1\,\big)$   \vspace{0.04cm}
   \end{tabular}     & 8 \\
    \hline
    %%%%%%%%%%%%%%%
           \begin{tabular}{c}  
  \vspace{-0.36cm}\\
${}^{}$ \hspace{-0.15cm}  $\big(\,2\,;\, 1^4\,\big)$   \vspace{0.04cm}
   \end{tabular}     & 70 \\
    \hline
    %%%%%%%%%%%%%%%
           \begin{tabular}{c}  
  \vspace{-0.36cm}\\
${}^{}$ \hspace{-0.15cm}  $\big(\,3\,;\, 2, 1^5\,\big)$   \vspace{0.04cm}
   \end{tabular}     &  168 \\
    \hline
    %%%%%%%%%%%%%%%
       \begin{tabular}{c}  
  \vspace{-0.36cm}\\
${}^{}$ \hspace{-0.15cm}  $\big(\,4\,;\, 2^3, 1^4\,\big)$   \vspace{0.04cm}
   \end{tabular}     & 280 \\
    \hline
    %%%%%%%%%%%%%%%
        \begin{tabular}{c}  
  \vspace{-0.36cm}\\
${}^{}$ \hspace{-0.15cm}  $\big(\,4\,;\, 3, 1^7\,\big)$   \vspace{0.04cm}
   \end{tabular}     & 8 \\
    \hline
    %%%%%%%%%%%%%%%
       \begin{tabular}{c}  
  \vspace{-0.36cm}\\
${}^{}$ \hspace{-0.15cm}  $\big(\,5\,;\, 2^6, 1\,\big)$   \vspace{0.04cm}
   \end{tabular}     & 56 \\
    \hline
    %%%%%%%%%%%%%%%
       \begin{tabular}{c}  
  \vspace{-0.36cm}\\
${}^{}$ \hspace{-0.15cm}  $\big(\,5\,;\, 3, 2^3, 1^4\,\big)$   \vspace{0.04cm}
   \end{tabular}     &  280 \\
    \hline
    %%%%%%%%%%%%%%%
        \begin{tabular}{c}  
  \vspace{-0.36cm}\\
${}^{}$ \hspace{-0.15cm}  $\big(\,6\,;\, 3^2, 2^4, 1^2\,\big)$   \vspace{0.04cm}
   \end{tabular}     &  420 \\
    \hline
    %%%%%%%%%%%%%%%
         \begin{tabular}{c}  
  \vspace{-0.36cm}\\
${}^{}$ \hspace{-0.15cm}  $\big(\,7\,;\, 3^4, 2^3, 1\,\big)$   \vspace{0.04cm}
   \end{tabular}     &  280 \\
    \hline
    %%%%%%%%%%%%%%%
            \begin{tabular}{c}  
  \vspace{-0.36cm}\\
${}^{}$ \hspace{-0.15cm}  $\big(\,7\,;\, 4, 3, 2^6\,\big)$   \vspace{0.04cm}
   \end{tabular}     &  56 \\
    \hline
    %%%%%%%%%%%%%%%
            \begin{tabular}{c}  
  \vspace{-0.36cm}\\
${}^{}$ \hspace{-0.15cm}  $\big(\,8\,;\, 3^7, 1\,\big)$   \vspace{0.04cm}
   \end{tabular}     &  8 \\
    \hline
    %%%%%%%%%%%%%%%
            \begin{tabular}{c}  
  \vspace{-0.36cm}\\
${}^{}$ \hspace{-0.15cm}  $\big(\,8\,;\, 4, 3^4, 2^3\,\big)$   \vspace{0.04cm}
   \end{tabular}     &  280 \\
    \hline
    %%%%%%%%%%%%%%%
            \begin{tabular}{c}  
  \vspace{-0.36cm}\\
${}^{}$ \hspace{-0.15cm}  $\big(\,9\,;\, 4^2, 3^5, 2 \,\big)$   \vspace{0.04cm}
   \end{tabular}     &  168 \\
    \hline
    %%%%%%%%%%%%%%%
             \begin{tabular}{c}  
  \vspace{-0.36cm}\\
${}^{}$ \hspace{-0.15cm}  $\big(\,10\,;\, 4^4, 3^4 \,\big)$   \vspace{0.04cm}
   \end{tabular}     &  70 \\
    \hline
    %%%%%%%%%%%%%%%
             \begin{tabular}{c}  
  \vspace{-0.36cm}\\
${}^{}$ \hspace{-0.15cm}  $\big(\,11\,;\, 4^7, 3 \,\big)$   \vspace{0.04cm}
   \end{tabular}     &  8 \\
    \hline
    %%%%%%%%%%%%%%%
\end{tabular}
\bk 
\end{center}
\caption{Types of lines and of conic classes and their numbers for the degree 1 del Pezzo surface $X_8$.}
\label{tab:Lines-and-Conics-on-X8}
\end{table}

\newpage

\newpage
%%%%%%%%%%%%%%%%%%%%%%%%%%%%%%%%%%%%%%%
%%%%%%%%%%%%%%%%%%%%%%%%%%%%%%%%%%%%%%%
\section{\bf The identity ${\bf HLog}^{r-2}$}
\label{S:Proof}
%%%%%%%%%%%%%%%%%%%%%%%%%%%%%%%%%%%%%%%
%%%%%%%%%%%%%%%%%%%%%%%%%%%%%%%%%%%%%%%
In the whole section,  we fix $3\leq r\leq 8$.  For most of the time, we will denote for simplicity %$X=X_r$, $\cL:=\cL_r$, $\cK:=\cK_r$. 
$$X=X_r\, ,\quad K=K_{X_r}\,, \quad \cL=\cL_r\, , \quad \cK=\cK_r\,, \quad {\rm etc.}$$

\subsection{}
For each conic class $\mathfrak c\in \cK$, we consider the corresponding conic fibration $\phi_\mc: X\ra \PP^1$ and we denote by $\Sigma_\mc\subset\PP^1$  the set of $r-1$ distinct points corresponding to the reducible fibers of $\phi_\mc$.
We may assume without loss of generality that $\Sigma_\mc=\{\sigma^1_\mc,\ldots, \sigma^{r-1}_\mc\}$ (with $\sigma^{r-1}_\mc=\infty$),  i.e., we are in the situation of a web of hypersurfaces 
(conics, in our case) as described in \S\ref{SS: webs}. We consider the same set-up and notations as in \S\ref{SS: webs}: 
we have
 %(where $\mc$ stands for an arbitrary conic class): 
%%%%%%%%%%%%%%%%%%%%%%%%%%%%%%%%%%%%%%%
\begin{itemize}
%%%%%%%%%%%%%%%%%%%%%%%%%%%%%%%%%%%%%%%
\item   $ Y=X\setminus L$ with  $L=L_r=\sum_{\ell\in\cL}\ell \subset X$, and 
$\HH={\bf H}^0\big(X,\Omega_X^1({\rm Log}\,L)\big)$; 
\sk
%%%%%%%%%%%%%%%%%%%%%%%%%%%%%%%%%%%%%%%
\item and for any conic class $\mc\in \cK$, we set: 
%\marginpar{\textcolor{red}{\bf (No more $\omega_{\mc}^i$)}}
\sk
%%%%%%%%%%%%%%%%%%%%%%%%%%%%%%%%%%%%%%%
\begin{itemize}
%%%%%%%%%%%%%%%%%%%%%%%%%%%%%%%%%%%%%%%
%%%%%%%%%%%%%%%%%%%%%%%%%%%%%%%%%%%%%%%
\item $
\HH_{\Sigma_\mc}={\bf H}^0\big(\PP^1,\Omega^1_{\PP^1}({\rm Log}\,\Sigma_\mc)\big)$ and $\HH_\mc=\phi_\mc^*\HH_{\Sigma_\mc}\subset \HH$;
\sk 
%\item $\omega^\mc_i={dz}/({z-\sigma_\mc^i})$ for $i=1,\ldots,r-2$;
%\sk
\item $\eta'_\mc=\wedge_{i=1}^{r-2} 
%{dz}/({z-\sigma_\mc^i})
\Big(\frac{dz}{z-\sigma_\mc^i}\Big)
\in\wedge^{r-2}\HH_{\Sigma_\mc}\subset\big(\HH_{\Sigma_\mc}\big)^{\otimes {r-2}}$; and 
%$\eta_\mc=\phi_\mc^*\eta'_\mc\in\wedge^{r-2}\HH_\mc\subset\HH_\mc^{\otimes {r-2}}$.
\sk 
\item  $\eta_\mc
=\wedge_{i=1}^{r-2} 
%{dz}/({z-\sigma_\mc^i})
\Big(\frac{d\phi_\mc}{\phi_\mc-\sigma_\mc^i}\Big)=\phi_\mc^*\eta'_\mc
\in\wedge^{r-2}\HH_\mc\subset\big(\HH_\mc\big)^{\otimes {r-2}}$.
%%%%%%%%%%%%%%%%%%%%%%%%%%%%%%%%%%%%%%%
\end{itemize}
%%%%%%%%%%%%%%%%%%%%%%%%%%%%%%%%%%%%%%%
\end{itemize}
%%%%%%%%%%%%%%%%%%%%%%%%%%%%%%%%%%%%%%%

Each of the elements $\eta_\mc$ generate the $1$-dimensional $\CC$-vector spaces $\wedge^{r-2}\HH_\mc$ and is canonically defined  up to  sign. 
In what follows, we  identify  $\wedge^{r-2}\HH_\mc$ with its image in $\wedge^{r-2}\HH$.

Using  the  notations of \S\ref{SS: webs}
 for any  $y\in Y$ and as holomophic germs at this point, 
%For any $y\in Y$, as holomophic germs at $y$ on $Y$ and using  Notation \ref{heavy notn}, 
one has 
%%%%%%%%%%%%%%%%%%%%%%%%%%%%%%%%%%%%%%%%%%
\begin{equation}
\label{Eq:AI-r-2}
%%%%%%%%%%%%%%%%%%%%%%%%%%%%%%%%%%%%%%%%%%
AI^{r-2}_{\mc}(\phi_c)=AI^{r-2}_{\Sigma_\mc,y}(\phi_\mc)={\rm II}_y(\eta_\mc)\in \cO_{Y,y}\, .
%%%%%%%%%%%%%%%%%%%%%%%%%%%%%%%%%%%%%%%%%%
\end{equation}
%%%%%%%%%%%%%%%%%%%%%%%%%%%%%%%%%%%%%%%%%%
 It then follows from Lemma \ref{Prop:Symbolic}, that Theorem \ref{Thm:MMain} 
 is equivalent to  the following statement:

%%%%%%%%%%%%%%%%%%%%%%%%%%%%%%%%%%%%%%%%%%
\begin{thm}
\label{Thm:main}
%%%%%%%%%%%%%%%%%%%%%%%%%%%%%%%%%%%%%%%%%%
{\rm 1.} Up to a global sign, there is a canonical choice of a tuple $(\tau_\mc)_{ \mc\in \cK}$ with  $\tau_\mc=\pm \eta_\mc$ for each  $\mc\in \cK$ and such that the following equality holds true in $\wedge^{r-2}\HH$: 
\begin{equation}
\label{Eq:Main-equation}
\sum_{\mc\in\cK_r} \tau_\mc=0
\end{equation}
{\rm 2.} Moreover, the identity \eqref{Eq:Main-equation} spans the space of linear relations between the  $\tau_\mc$'s, i.e., if $(c_\mc)_{\mc\in \cK}\in \mathbf C^{\cK}$ is such that $\sum_{\mc\in\cK_r}c_\mc \, \tau_\mc=0$ then all the $c_\mc$'s are equal. 
%%%%%%%%%%%%%%%%%%%%%%%%%%%%%%%%%%%%%%%%%%
\end{thm}
%%%%%%%%%%%%%%%%%%%%%%%%%%%%%%%%%%%%%%%%%%

The rest of this section is devoted to proving this result.

%%%%%%%%%%%%%%%%%%%%%%%%%%%%%%%%%%%%%%%%%%
\subsection{} 
\label{SS:SS-32}
%%%%%%%%%%%%%%%%%%%%%%%%%%%%%%%%%%%%%%%%%%
The irreducible components of  $L$ being 
the lines $\ell \in\cL$, one can define a Poincar\'e residue map 
${\rm Res}_L=\oplus_{\ell\in\cL}{\rm Res}_\ell: \Omega_X^1\big({\rm Log}\,L\big)\ra\oplus_{\ell \in\cL}\cO_\ell$ 
wich makes   the following
 sequence of sheaves exact: 
 %%%%%%%%%%%%%%%%%%%%%%%%%%%%%%%%%%%%%%%%%%
$$0\ra\Omega_X^1\longmapsto\Omega_X^1\big({\rm log} \,L\big)\longmapsto\oplus_{\ell\in\cL}\cO_\ell\ra0.$$
%%%%%%%%%%%%%%%%%%%%%%%%%%%%%%%%%%%%%%%%%%
As $X$ is a rational variety, we have that ${\bf H}^0(X,\Omega_X^1)=0$, hence the residue map induces an injective map of $\CC$-linear vector spaces
${\rm Res}_L: \HH\hra\CC^{\cL}$ and in turn an injective linear map $\wedge^{r-2}\HH\hra\wedge^{r-2}\CC^{\cL}$.
%%%%%%%%%%%%%%%%%%%%%%%%%%%%%%%%%%%%%%%%%%
\sk

Given a conic fibration $\phi_\mc: X\ra\PP^1$ associated to a conic class $\mc\in\cK$,  we denote by $C_{ \mc}^1,\ldots, C_{ \mc}^{r-1}$ the reducible fibers of $\phi_\mc$, with $C_{ \mc}^i=\phi_\mc^{-1}(\sigma^i_{\mc})$ for $i=1,\ldots,r-1$ (with $\sigma^{r-1}_\mc=\infty$).
Each conic $C_{ \mc}^i$ is a union of two lines $\ell_{ \mc}^i$, $\tilde \ell_{ \mc}^i$ intersecting in one point. 
It follows that the residues
%%%%%%%%%%%%%%%%%%%%%%%%%%%%%%%%%%%%%%%%%%
$${\rm Res}_L
\left(  {d  \phi_\mc}/\Big({ \phi_\mc-\sigma_{\mc}^i \Big)}
\right)
=  
C_{ \mc}^i-C_{ \mc}^{r-1}=\Big(\, \ell_{ \mc}^i+\tilde \ell_{ \mc}^i\, \Big)-\Big(
\, \ell_{ \mc}^{r-1}+\tilde \ell_{ \mc}^{r-1}\, \Big)\in \CC^\cL$$
%%%%%%%%%%%%%%%%%%%%%%%%%%%%%%%%%%%%%%%%%%
for $ i=1,\ldots, r-2$, form a basis for the image of 
$\HH_\mc\subset\HH$ under the injective map $\HH\hra\CC^{\cL}$. Consequently,  we get that the image of 
$\eta_\mc\in\wedge^{r-2}\HH_\mc\subset\wedge^{r-2}\HH$ under the injective map 
$\wedge^{r-2}\HH\hra\wedge^{r-2}\CC^{\cL}$ is 
%%%%%%%%%%%%%%%%%%%%%%%%%%%%%%%%%%%%%%%%%%
\begin{equation}
\label{Eq:(r-2)-Wedge-Cci}
\Big(C_{ \mc}^1-C_{ \mc}^{r-1}\Big)\wedge\ldots\wedge\Big (C_{ \mc}^{r-2}-C_{ \mc}^{r-1}\Big).\end{equation}
%%%%%%%%%%%%%%%%%%%%%%%%%%%%%%%%%%%%%%%%%%

%%%%%%%%%%%%%%%%%%%%%%%%%%%%%%%%%%%%%%%%%%
%%%%%%%%%%%%%%%%%%%%%%%%%%%%%%%%%%%%%%%%%%
\subsection{}
\label{SS:SS-33}
%%%%%%%%%%%%%%%%%%%%%%%%%%%%%%%%%%%%%%%%%%
%%%%%%%%%%%%%%%%%%%%%%%%%%%%%%%%%%%%%%%%%%
The Weyl group $W$ acts on the set of lines $\cL$ and on the set of conic classes $\cK$ 
in a compatible way. In particular, for $\mc \in \cK$ given, the action of any $w\in W$ 
 sends  the reducible fibers of $\phi_\mc$ to the reducible fibers of $\phi_{w\cdot\mc}$ in the following way: 
$$w\cdot C_\mc^{\,i} =w\cdot \ell_{\mc}^{\,i}+w\cdot \tilde \ell_{\mc}^{\,i}\,. $$

On the other hand, the action of $W$ on $\cL$ induces a canonical linear action of $W$ on $\CC^\cL$, and therefore on  $\wedge^{r-2}\CC^\cL$.  The action of  $w\in W$ on 
any wedge product $\wedge_{i=1}^{r-2} \ell_i
%\ell_1\wedge \cdots \wedge \ell_{r-2} 
$
 with $(\ell_i)_{i=1}^{r-2}\in \cL^{r-2}$ is given by 
\begin{equation*}
w\cdot \Big(
\ell_1\wedge \cdots \wedge \ell_{r-2} \Big)= 
\big(w\cdot \ell_1\big) \wedge \cdots \wedge\big(w\cdot \ell_{r-2}\big) \, . 
\end{equation*}

%%%%%%%%%%%%%%%%%%%%%%%%%%%%%%%%%%%%%%%%%%
%%%%%%%%%%%%%%%%%%%%%%%%%%%%%%%%%%%%%%%%%%
\subsection{}
%%%%%%%%%%%%%%%%%%%%%%%%%%%%%%%%%%%%%%%%%%
%%%%%%%%%%%%%%%%%%%%%%%%%%%%%%%%%%%%%%%%%%
We now fix a base conic class $\mc_1=h-e_1$ and label the reducible fibers of the associated conic fibration
$\phi_{\mc_1}: X\ra\PP^1$ by 
$${}^{}\hspace{3cm} C^{\,i}=C_{\mc_1}^{\,i-1}=l_{1i}+e_i\qquad \mbox{ for }\quad  i=2,\ldots, r\, ,$$ where 
$l_{1i}$ stands fo the class of the strict transform under the blow-up map $\beta$ of the line in $\PP^2$ through $p_1$ and $p_i$, i.e.,  $l_{1i}=h-e_1-e_i$.
As a generator of (the image in $\wedge^{r-2}\CC^{\cL}$ of) $ \wedge^{r-2}\HH _{\mc_1}$, 
we choose and fix 
$$ \tau_{\mc_1}=\Big(C^2-C^{r}\Big)\wedge \Big(C^3-C^{r}\Big)
\wedge\ldots\wedge\Big(C^{r-1}-C^{r}\Big)\in 
%\wedge^{r-2}\HH _{\mc_1}\subset
\wedge^{r-2}\CC^{\cL}.$$

For $w\in W$  arbitrary,  we have:
\begin{equation}
\label{Eq:bullet}
%%%%%%%%%%%%%%%%%%%%%%%%%%%%%%%%%%
w\cdot  \tau_{\mc_1}= 
\big(w\cdot C^2-w\cdot C^{r}\big)\wedge\big(w\cdot C^3-w\cdot C^{r}\big)\wedge\ldots\wedge \big(w\cdot C^{r-1}-w\cdot C^{r}\big)\in 
%\wedge^{r-2}M_{w\cdot \mc_1}\subset
\wedge^{r-2}\CC^{\cL}. 
\end{equation}

The stabilizer $W_{\mc_1}$ of $\mc_1$ is a subgroup  of $W$ hence naturally acts on
 $\wedge^{r-2}\CC^{\cL}$. This action lets 
 $\wedge^{r-2} \HH_{\mc_1}\subset \wedge^{r-2}\CC^{\cL}$ invariant hence  
 $\wedge^{r-2} \HH_{\mc_1}$ is naturally a $W_{\mc_1}$-representation (of dimension 1).

%%%%%%%%%%%%%%%%%%%%%%%%%%%%%%%%%%
\begin{lem} 
\label{Lem:Lolo}
%%%%%%%%%%%%%%%%%%%%%%%%%%%%%%%%%%
{\rm 1.} As a $W_{\mc_1}$-representation, $\wedge^{r-2} \HH_{\mc_1}$ is isomorphic 
to the signature representation.  
%{\it i.e}\,for any $w\in W_{\mc_1}$, one has  $w\cdot \tau_{\mc_1}=(-1)^w \tau_{\mc_1}$. \sk

\noindent {\rm 2.} For $w\in W$ and $\mc\in\cK$ such that $\mc=w\cdot\mc_1$, the element 
$$\tau_{\mc}=(-1)^w\big (w \cdot  \tau_{\mc_1}\big)$$
is a well defined generator of  $\wedge^{r-2} \HH_{\mc}\subset \wedge^{r-2}\CC^{\cL}$.
%%%%%%%%%%%%%%%%%%%%%%%%%%%%%%%%%%
\end{lem}
%%%%%%%%%%%%%%%%%%%%%%%%%%%%%%%%%%

 With the notation of this lemma, since the $w\cdot C^i$'s for $i=2,\ldots,r-1$ are the non irreducible fibers of $\phi_\mc$, one clearly has that $\tau_{\mc}$ coincides with $\eta_\mc$ up to sign, hence, in particular, is a generator of $\wedge^{r-2} \HH_{\mc}$. The interest of the second statement  in this lemma is that it asserts that $\tau_{\mc}$ only depends on $\mc$ and not on $w$ (once $\tau_{\mc_1}$ has been fixed). 

%%%%%%%%%%%%%%%%%%%%%%%%%%%%%%%%%%
\begin{proof}
%%%%%%%%%%%%%%%%%%%%%%%%%%%%%%%%%%
Proving 1.\,is elementary. Indeed,  being a $W_{\mc_1}$-representation of dimension 1, there are only two possibilities for $\wedge^{r-2} \HH_{\mc_1}$: either it is the trivial 
$W_{\mc_1}$-representation, or it is the signature representation. To prove that the second case does occur, it suffices to exhibit an element  $w\in W_{\mc_1}$ such that $w\cdot \tau_{\mc_1}=-\tau_{\mc_1}$.
 Using  \eqref{Eq:s-alpha-i} and \eqref{Eq:bullet}, 
 it is straightforward to check 
  that any of the generators $ s_2,\ldots,s_r$ of $W_{\mc_1}$ ({\it cf.}\,\S\ref{SS:dPd}.(12)) has this property.

The second part of the lemma follows easily from the first: for $w_1,w_2\in W$ such that $w_1\cdot \mc_1=w_2\cdot \mc_1=\mc$, one has $w_2^{-1}w_1\in W_{\mc_1}$,  hence 
$w_2^{-1}w_1\cdot \tau_{\mc_1}=(-1)^{w^{-1}_2w_1} \tau_{\mc_1}$ by {\rm 1.} Thus 
$w_2\cdot \big(w_2^{-1}w_1\cdot \tau_{\mc_1}\big)=(-1)^{w^{-1}_2w_1}\,
w_2\cdot \tau_{\mc_1}$. The signature $ w\mapsto (-1)^w$ being a group morphism, one has  $(-1)^{w^{-1}}=(-1)^w$ for any $w$, therefore 
one obtains that $w_1\cdot \tau_{\mc_1}=(-1)^{w_2} (-1)^{w_1}\,
w_2\cdot \tau_{\mc_1}$. This is equivalent to  $(-1)^{w_1} w_1\cdot \tau_{\mc_1}=
(-1)^{w_2} \,
w_2\cdot \tau_{\mc_1}$, which is the relation ensuring that 2.\,holds true.
%%%%%%%%%%%%%%%%%%%%%%%%%%%%%%%%%%
\end{proof}
%%%%%%%%%%%%%%%%%%%%%%%%%%%%%%%%%%

%%%%%%%%%%%%%%%%%%%%%%%%%%%%%%%%%%
%%%%%%%%%%%%%%%%%%%%%%%%%%%%%%%%%%
\subsection{Proof of Theorem \ref{Thm:main}} 
\label{SS:zogo}
%%%%%%%%%%%%%%%%%%%%%%%%%%%%%%%%%%
%%%%%%%%%%%%%%%%%%%%%%%%%%%%%%%%%%
We are going to prove that the following statements are satisfied: 
%%%%%%%%%%%%%%%%%%%%%%%%%%%%%%%%%%
\begin{enumerate}
%%%%%%%%%%%%%%%%%%%%%%%%%%%%%%%%%%
%%%%%%%%%%%%%%%%%%%%%%%%%%%%%%%%%%
\item[{\rm 1.}] {\it The $\cK$-tuple $\big(\tau_\mc\big)_{ \mc\in \cK}$ is a basis of $\oplus_{ \mc\in \cK}
\wedge ^{r-2}\HH_{\mc}$ which is canonical, up to a global sign.}\sk
%%%%%%%%%%%%%%%%%%%%%%%%%%%%%%%%%%
\item[{\rm 2.}]  
{\it  The sum $\sum_{ \mc\in \cK} \tau_\mc$ in $\wedge^{r-2} \HH$
%is $W$-invariant and 
transforms as the signature under the action of $W$.}\sk 
%%%%%%%%%%%%%%%%%%%%%%%%%%%%%%%%%%
\item[{\rm 3.}] {\it One has $\sum_{ \mc\in \cK} \tau_\mc=0 $ in $\wedge^{r-2} \HH$.}
\sk
%%%%%%%%%%%%%%%%%%%%%%%%%%%%%%%%%%
\item[{\rm 4.}] {\it Any scalar linear relation between the $\tau_\mc $'s in $\wedge^{r-2} \HH$ is a multiple of the one corresponding to the identity of 3.}
\sk
%%%%%%%%%%%%%%%%%%%%%%%%%%%%%%%%%%
%%%%%%%%%%%%%%%%%%%%%%%%%%%%%%%%%%
\end{enumerate}
%%%%%%%%%%%%%%%%%%%%%%%%%%%%%%%%%%
%%%%%%%%%%%%%%%%%%%%%%%%%%%%%%%%%%

The first assertion follows easily from the second part of Lemma \ref{Lem:Lolo} and from the fact that each $\tau_\mc$ necessarily coincides with $\eta_\mc$ up to sign (details are left to the reader). In  this subsection, we are going to establish  first 2.\,({\it cf.} Lemma \ref{Lem:hlog-signature}) then 3.\,and 4.\,which will follow in the same time from Lemma \ref{Lem:gno-gno}.
\sk 

 For $k\geq 1$, following Manin ({\it cf.}\,\cite[\S26]{Manin}),  we call an {\it exceptional $k$-tuple} any $k$-tuple $(\ell_i)_{i=1}^k\in \cL^k$ of non intersecting lines, i.e., such that $\ell_i\cdot \ell_j=0$ for any $i,j$ such that $1\leq i<j\leq k$. Our approach to prove both 2.\,and 3.\,together is elementary and relies on the following 

%%%%%%%%%%%%%%%%%%%%%%%%%%%%%%%%%%%%%%%
\begin{lem}
\label{Claim:E-exceptional-(r-2)-tuple}
%%%%%%%%%%%%%%%%%%%%%%%%%%%%%%%%%%%%%%%
%{\bf Claim.}  
{\it Let $\mathcal E$ be an exceptional $(r-2)$-tuple. 
There are exactly two conic classes such that each element of $\mathcal E$ appears as a component of a reducible fiber of the associated conic fibration.
% $\phi_{\mathfrak c_{\mathcal E}}$ and $\phi_{\mathfrak c'_{\mathcal E}}$. 
}
%%%%%%%%%%%%%%%%%%%%%%%%%%%%%%%%%%%%%%%
\end{lem}
%%%%%%%%%%%%%%%%%%%%%%%%%%%%%%%%%%%%%%%
\begin{proof} 
Since $W$ acts transitively on the set of exceptional $(r-2)$-tuples 
(according to \S\ref{SS:dPd}.(13)), one can assume that $\mathcal E=(e_3,\ldots,e_{r})$.  Concretely, one wants to determine
the conic classes $\mathfrak c\in \cK $ such that 
$\mathfrak c- e_i \in \cL$ for $i=3,\ldots,r$. 
Since the two sets $\boldsymbol{\mathcal L}$ and $\boldsymbol{\mathcal K}$ are finite and 
can be explicitly described  (See Table \ref{tab:Lines-and-Conics-on-X8}  above for the case $r = 8$), the claim can be checked by a straightforward case by case verification.  
One finds that only the two conic classes $\mathfrak c_{1}=h-e_{1} $ and $ 
\mathfrak c_{2}=h-e_2 $ satisfy the above conditions.
\end{proof}

We consider the following element of $ \wedge^{r-2} \CC^{\cL}$:
$${\bf hlog}={\bf hlog}^{r-2}=\sum_{ \mc\in \cK} \tau_\mc\, .$$ 

 We now prove that this element is equal to zero, by decomposing it in the canonical basis of $ \wedge^{r-2} \CC^{\cL}$ given by the wedge products $\ell_1\wedge \ell_2\wedge\ldots \wedge \ell_{r-2}$ of $r-2$ pairwise distinct lines $\ell_1,\ldots,\ell_{r-2}\in \cL$.

Let $ \wedge^{r-2}_{exc} \CC^{\cL}$ be the proper subspace of $ \wedge^{r-2} \CC^{\cL}$ 
spanned by the wedge products $\ell_1\wedge \ell_2\wedge\ldots \wedge \ell_{r-2}$ for all 
exceptional $(r-2)$-tuple of lines $(\ell_i)_{i=1}^{r-2}$. 
We call such wedge products  {\it exceptional}. 
We fix once for all a basis of $ \wedge^{r-2}_{exc} \CC^{\cL}$ 
formed by exceptional wedge products and call it the {\it exceptional basis} of $ \wedge^{r-2}_{exc} \CC^{\cL}$. Note that this basis  is unique but only
 up to changing the signs of its elements. 
However, the arguments below are not essentially affected by this ambiguity.

We may assume that  $\wedge_{i=3}^r   e_i$ is an element of the exceptional basis and we denote by 
\begin{equation}\label{Lambda}
\Lambda : 
  \wedge^{r-2}_{exc} \CC^{\cL} \longrightarrow \CC 
 \end{equation}
  the associated linear coordinate form with respect to the exceptional basis we are working with,  i.e.,  $\Lambda$ is the linear form 
on  $\wedge^{r-2}_{exc} \CC^{\cL}$  
  characterized by the relations 
  $\Lambda( \wedge_{i=3}^r   e_i)=1$ and 
$\Lambda( \wedge_{i=3}^r   \ell_i)=0$  for any exceptional wedge product 
$\wedge_{i=3}^r   \ell_i $ such that $\{\ell_i\}_{i=3}^r \neq \{ e_i\}_{i=3}^r$.

We first prove the
%%%%%%%%%%%%%%%%%%%%%%%%%%%%%%%%%%%%%%%
\begin{lem} 
\label{Lem:In-wedge-exc}
%%%%%%%%%%%%%%%%%%%%%%%%%%%%%%%%%%%%%%%
For any $\mc\in \cK$,  $\tau_\mc$ belongs to $  \wedge^{r-2}_{exc} \CC^{\cL}$, therefore ${\bf hlog}=\sum_{ \mc\in \cK} \tau_\mc$ as well.
%%%%%%%%%%%%%%%%%%%%%%%%%%%%%%%%%%%%%%%
\end{lem}
%%%%%%%%%%%%%%%%%%%%%%%%%%%%%%%%%%%%%%%
%%%%%%%%%%%%%%%%%%%%%%%%%%%%%%%%%%%%%%%
%%%%%%%%%%%%%%%%%%%%%%%%%%%%%%%%%%%%%%%
\begin{proof} 
%%%%%%%%%%%%%%%%%%%%%%%%%%%%%%%%%%%%%%%
For $\mc\in \cK$, $\tau_\mc$ is equal to  \eqref{Eq:(r-2)-Wedge-Cci} up to sign. 
The lemma follows easily by noticing that  the conics $C_{\mc}^1,\ldots,C_{\mc}^{r-1}$  are pairwise disjoint and because each of them is the sum of two lines. 
%Since $ \wedge^{r-2}_{exc} \CC^{\cL}$ is stable under the action of $W$, 
%the claim follows by remarking that the statement is obvious in the specific case of  
%$\mc_1=h-e_1$. 
%%%%%%%%%%%%%%%%%%%%%%%%%%%%%%%%%%%%%%%
\end{proof} 
%%%%%%%%%%%%%%%%%%%%%%%%%%%%%%%%%%%%%%%

We also need  the following 
%%%%%%%%%%%%%%%%%%%%%%%%%%%%%%%%%%%%%%%
\begin{lem} 
\label{Lem:hlog-signature}
%%%%%%%%%%%%%%%%%%%%%%%%%%%%%%%%%%%%%%%
%The element  ${\bf hlog}$ of $\wedge^{r-2}\CC^{\cL_r}$  is 
%signature-invariant, i.e., 
For any $w\in W$, one has
$ 
w\cdot {\bf hlog}=(-1)^w \,{\bf hlog}
$.
%%%%%%%%%%%%%%%%%%%%%%%%%%%%%%%%%%%%%%%
\end{lem}
%%%%%%%%%%%%%%%%%%%%%%%%%%%%%%%%%%%%%%%
\noindent{\it Proof.} 
%%%%%%%%%%%%%%%%%%%%%%%%%%%%%%%%%%%%%%%
 For a conic class $\mc\in\cK$, let $w_\mc \in W$ be such that $\mc=w_\mc \cdot\mc_1$. By the definition of $\tau_\mc$, one has
$$\tau_\mc
= (-1)^{w_\mc} w_\mc\cdot \tau_{\mc_1}
= (-1)^{w_\mc} \wedge_{i=2}^{r-1} \Big(w_\mc\cdot C^i-w_\mc \cdot C^{r}\Big)\,.
$$ 
 
Let $w$ be an arbitrary element of $W$. 
From \eqref{Eq:bullet} and because $ww_\mc\cdot \mc_1=w\cdot\mc$, it comes that 
%  $w\cdot\tau_\mc
%= (-1)^{w_\mc} \wedge_{i=1}^{r-2} \big(ww_\mc\cdot C^i-ww_\mc \cdot %C^{r-1}\big)
%=(-1)^{w}\tau_{w\cdot\mc}
%$ for any $w\in W$.
$$w\cdot\tau_\mc
= (-1)^{w_\mc} \wedge_{i=2}^{r-1} \Big(ww_\mc\cdot C^i-ww_\mc \cdot C^{r}\Big)
=(-1)^{w}\tau_{w\cdot\mc}
\,.
$$
%%%%%%%%%%%%%%%%%%%%%%%%%%%%%%%%%%%%%%%
Summing up on the conic classes and because $\mc\mapsto w\cdot \mc$ is a bijection of $\cK$, one gets 
%%%%%%%%%%%%%%%%%%%%%%%%%%%%%%%%%%%%%%%
$$
 {}^{} \hspace{3.5cm}
w\cdot {\bf hlog}= \sum_{\mc\in\cK}w\cdot \tau_\mc=(-1)^w\sum_{\mc\in\cK}\tau_{w\cdot\mc}=(-1)^w\, {\bf hlog}\, .  \hspace{3.3cm}\qed
$$

 Since the $W$-action on $\Pic(X)$ preserves the intersection product, 
$\wedge^{r-2}_{exc} \CC^{\cL}$ is a proper $W$-submo\- -dule of $\wedge^{r-2} \CC^{\cL}$. Furthermore, $W$ acts on $\wedge^{r-2}_{exc} \CC^{\cL}$  
by permuting the elements of the exceptional basis. 
Moreover, by \S\ref{SS:dPd}.(13) this action 
on the exceptional wedge products is transitive (up to sign). 
%{\red{\bf (Which one??)}}\footnote{\red{There two actions of $W$here: the one on $\wedge^{r-2} \CC^{\cL}$,  the other (inducing the former) on the exceptional wedge product elements of the exceptional basis. Maybe we should be more precise here?}}
 To check that ${\bf hlog}=0$ in  $\wedge^{r-2}_{exc} \CC^{\cL}$, it suffices to check that any element
$\wedge_{i=1}^{r-2}\ell_i$ in the exceptional basis appears in ${\bf hlog}$ with coefficient zero. For such 
$\wedge_{i=1}^{r-2}\ell_i$, let  $w\in W$ be such that 
$w\cdot(\wedge_{i=3}^r  e_i)=\wedge_{i=1}^{r-2}\ell_i$. It suffices to check that $\Lambda(w^{-1}\cdot {\bf hlog})=0$ (see notation (\ref{Lambda}). 
By Lemma \ref{Lem:hlog-signature}, this would follow from verifying  that $\Lambda({\bf hlog})=0$.

From Lemma 
\eqref{Claim:E-exceptional-(r-2)-tuple},  the only conic classes  $\mc\in \cK$ 
 for which 
$\wedge_{i=3}^r e_i $ appears with non-zero coefficient in the decomposition of $\tau_\mc$ in the exceptional basis are 
$\mc_{1}=h-e_{1}$ and $
\mc_2=h-e_2$, i.e.,
\begin{equation}
\label{Eq:zzz}
 \Big\{\, \mc\in \cK\, \,\big\lvert \, \,  \Lambda(\mc)\neq 0\, \Big\}= \big\{\, 
\mc_{1}\, , \, \mc_{2}\, \big\}\, . 
\end{equation}

 Consequently, one has 
\begin{equation}
\label{Eq:koko}
\Lambda\big(
{\bf hlog}\big)=
 \sum_{\mc\in\cK} \Lambda\big( \tau_\mc \big)=
\Lambda\big(
 \tau_{\mc_{1}}\big)+\Lambda\big(
 \tau_{\mc_{2}}\big)
\, .
\end{equation}

On the other hand, considering our initial choice for $\tau_{\mc_1}$, we have (see \S3.2 above) 
$$
\tau_{\mc_1}=\Big(  \ell_{12}+e_2-\ell_{1r}-e_r\Big)
\wedge 
\Big( \ell_{13}+e_3-\ell_{1r}-e_r \Big)\wedge \ldots
\wedge \Big(\ell_{1\,r-1}+e_{r-1}-\ell_{1r}-e_r \Big) 
$$
from which we have immediately that 
\begin{equation}
\label{Eq:Lambda(tau-c1)}
\Lambda \big(  \tau_{\mc_{1}} \big) =
\Lambda\Big(  -e_r \wedge e_3 \wedge e_4\wedge \ldots \wedge e_{r-1}\Big)=(-1)^r\, .
\end{equation}

From \S\ref{SS:dPd}.(7), we know that 
 the first fundamental reflection $s_1$ acts as the transposition exchanging $e_1$ and $e_2$ on the set $\{ h, e_1,\ldots,e_r\}$.  
Therefore, one has  $s_1\cdot \mc_1=\mc_2$  
 and $s_1\cdot \Big( e_3\wedge\ldots\wedge e_{r}\Big)=
\big( s_1\cdot e_3\big) \wedge\ldots\wedge \big( s_1\cdot e_r\big)
= e_3\wedge\ldots\wedge e_{r}
$ 
which implies that
$
\Lambda \big( s_1\cdot 
 \tau_{\mc_{1}}\big)=
 \Lambda \big(
 \tau_{\mc_{1}}\big)%=(-1)^r
$.
%%%%%%%%%%%%%%%%%%%%%%%%%%%%%%%%%%%%%%%%
%%%%%%%%%%%%%%%%%%%%%%%%%%%%%%%%%%%%%%%%
  Since $\tau_{\mc_2}=(-1)^{s_1} s_1\cdot \tau_{\mc_1}$ by definition and because $(-1)^{s_1}=-1$, one obtains that 
\begin{equation}
\label{Eq:Lambda(tau-c2)}
\Lambda( 
 \tau_{\mc_{2}})=(-1)^{r-1}=- \Lambda( 
 \tau_{\mc_{1}})\, .
\end{equation}  
%$
%\Lambda( 
%\tau_{\mc_{2}})=(-1)^{r-1}=- \Lambda( 
%\tau_{\mc_{1}})$ 

Substituting \eqref{Eq:Lambda(tau-c1)} and \eqref{Eq:Lambda(tau-c2)} in \eqref{Eq:koko} gives 
$\Lambda({\bf hlog})=0$ which, as explained above,  implies the
%%%%%%%%%%%%%%%%%%%%%%%%%%%%%%%%%%%%%%%%%
\begin{lem}
\label{Lem:gno-gno}
One has ${\bf hlog}=0$. 
%%%%%%%%%%%%%%%%%%%%%%%%%%%%%%%%%%%%%%%%%
\end{lem}

The assertion 3.\,at the beginning of \S\ref{SS:zogo} is proved. 
We now prove assertion 4. Consider the graph with vertices in $\cK$, with $\mc, \mc'\in\cK$ joined by an edge 
if there exists an element $\wedge_{i=1}^{r-2}\ell_i$ of the exceptional basis that appears with non-zero coefficient in the decomposition of both 
$\tau_\mc$ and $\tau_{\mc'}$ in the exceptional basis. We have

%%%%%%%%%%%%%%%%%%%%%%%%%%%%%%%%%%%%%%%%%
\begin{lem}
\label{Lem:gno-gno}
%%%%%%%%%%%%%%%%%%%%%%%%%%%%%%%%%%%%%%%%%
For any element $\ell_1\wedge \ldots \wedge \ell_{r-2}$ of the exceptional basis,  there are exactly two conic classes  $\mc,\mc'\in \cK$ such that $\ell_1\wedge \ldots \wedge \ell_{r-2}$ appears with non-zero coefficient in the decomposition of $\tau_\mc$  and $\tau_{\mc'}$ in the exceptional basis. 
Moreover, these two coefficients are opposite. 
\end{lem}
%%%%%%%%%%%%%%%%%%%%%%%%%%%%%%%%%%%%%%%%%
\begin{proof}
%%%%%%%%%%%%%%%%%%%%%%%%%%%%%%%%%%%%%%%%%
For $e_3\wedge\ldots\wedge e_r$, this has been proved above ({\it cf.}\,\eqref{Eq:zzz}, \eqref{Eq:Lambda(tau-c1)} and \eqref{Eq:Lambda(tau-c2)}). The general case follows by considering the action of $W$. 
%%%%%%%%%%%%%%%%%%%%%%%%%%%%%%%%%%%%%%%%%
\end{proof}
%%%%%%%%%%%%%%%%%%%%%%%%%%%%%%%%%%%%%%%%%

Assume now that there exists a relation $\sum_{c\in\cK} c_\mc\tau_\mc=0$ for some $c_\mc\in\CC$. 
By Lemma \ref{Lem:gno-gno}, for two conic classes $\mc$ and $\mc'$ connected by an edge, we have $c_\mc+c_{\mc'}=0$. 
Hence it suffices to prove that our graph is connected. As the action of $W$ on $\cK$ is transitive, it suffices to check that for all 
$w\in W$ the classes $\mc_1=h-e_1$ and $w\cdot \mc_1$ are connected by a sequence of edges. Furthermore, using again the action of $W$, it suffices to check this for $w=s_i$ ($i=1,\ldots,r$).
The reflections $s_2,\ldots, s_r$ belong to the stabilizer of $\mc_1$, so there is nothing to prove. If $w=s_1$, then $\mc_2=w\cdot \mc_1=h-e_2$. 
As proved above, $\mc_1$ and $\mc_2$ are connected by an edge.  
%as the element $\wedge_{i=3}^{r}e_i$ appears with non-zero coefficient in the decomposition of both 
%$\tau_\mc$ and $\tau_{\mc'}$ in the exceptional basis.
The proof of Theorem \ref{Thm:main} is now complete.
\mk

%%%%%%%%%%%%%%%%%%%%%%%%%%%%%%%%%%%%%%%
%%%%%%%%%%%%%%%%%%%%%%%%%%%%%%%%%%%%%%%
\subsection{A representation-theoretic interpretation}
\label{SS:representation-theoretic-interpretation}
%%%%%%%%%%%%%%%%%%%%%%%%%%%%%%%%%%%%%%%
%%%%%%%%%%%%%%%%%%%%%%%%%%%%%%%%%%%%%%%
The $\boldsymbol{\mathcal K}$-tuple $\tau_{\cK}=\big( \tau_{\mathfrak c}\big)_{{\mathfrak c} \in \cK}$ is an algebraic avatar 
of  the $\kappa_r$-tuple of hyperlogarithms  $\big( \epsilon_i\,AI_i^{r-2}(U_i)\big)_{i=1}^{\kappa_r}$ involved in the statement of Theorem \ref{Thm:MMain}. 
 It turns out that $\tau_{\cK}$  as well as 
the fact that the identity 
\begin{equation}
\label{Eq:hlog-r-2}
{\bf hlog}^{r-2}
%\stackrel{def}{=}
=
\sum_{ {\mathfrak c} \in \cK } \tau_{\mathfrak c}=0
\end{equation}
%its components satisfy the identity $\sum_{ {\mathfrak c} \in \boldsymbol{\mathcal K}_r } \Omega_{\mathfrak c}=0$ 
is satisfied in $\wedge^{r-2} \CC^\cL$ can be interpreted within the representation theory of the Weyl goup $W$.  This subsection is devoted to an exposition of this. 
For details, we refer to \cite{PirioChat} where the second author used this approach 
to give a representation theoretic proof of Theorem \ref{Thm:main}
 for del Pezzo surfaces of degree $d\in \{2,\ldots,6\}$.\sk
 \newpage
 
 The key points from this perspective are the following ({\it cf.}\,\cite[\S3.2]{PirioChat} for details): 
 %%%%%%%%%%%%%%%%%%%%%%%%%%%%
 \begin{enumerate}
 %%%%%%%%%%%%%%%%%%%%%%%%%%%%
 \item[$1.$] one can define  a natural action of  $W$ on the direct sum $\oplus_{{\mathfrak c} \in \cK } \wedge^{r-2} {\HH}_{\mathfrak c} $ such that the   map $$\iota_{\cK} : 
 \oplus_{{\mathfrak c} \in \cK } \wedge^{r-2} {\HH}_{\mathfrak c} 
 \longrightarrow \wedge^{r-2} \CC^{\cL}$$ induced by the natural inclusion 
 $\wedge^{r-2} {\HH}_{\mathfrak c} \hookrightarrow \wedge^{r-2} {\HH} \stackrel{{\rm Res}_L}{\longrightarrow}  \wedge^{r-2} \CC^{\cL}$  becomes a morphism of $W$-representations;
 \mk 
 %%%%%%%%%%%%%%%%%%%%%%%%%%%%
% \vspace{-0.35cm}
 \item[2.] as a $W$-representation, $\oplus_{{\mathfrak c} \in \cK } \wedge^{r-2} {\HH}_{\mathfrak c} $ is isomorphic to ${\bf sign} \otimes \mathbf C^{\cK}$  
 where ${\bf sign}$ stands for the signature $W $-representation and where 
 the $W$-module structure on $\mathbf C^{\cK}$ is the one induced by the action  of $W$ on $\cK$ by permutations;  \mk 
 %%%%%%%%%%%%%%%%%%%%%%%%%%%%
 \item[3.] the span of  $\tau_{\cK}=\big( \tau_{\mathfrak c}\big)_{{\mathfrak c} \in \cK}$ is $W$-invariant and is the unique 1-dimensional irreducible component of $\oplus_{{\mathfrak c} \in \cK} \wedge^{r-2} {\HH}_{\mathfrak c}$ which is isomorphic to the signature representation ${\bf sign}$;\mk
  %%%%%%%%%%%%%%%%%%%%%%%%%%%%
 \item[4.]
 from  1.\,and 3.\,it follows that $ 
  \iota_{\cK}\big( 
  \tau_{\cK}
  \big) =
 \sum_{\mathfrak c} \tau_{\mathfrak c} 
  = {\bf hlog}^{r-2}
 $ spans a $W$-subrepresentation of $\wedge^{r-2}\mathbf C^{\cL}$ which either is zero or is isomorphic to 
 ${\bf sign}$;\mk
  %%%%%%%%%%%%%%%%%%%%%%%%%%%%
 \item[5.] the decomposition of $\wedge^{r-2}\mathbf C^{\cL}$ in $W$-irreducibles can be determined explicitly (by means of computations with GAP). In particular, ${\bf sign}$ appears with positive multiplicity in this decomposition if and only if $r=8$.
%%%%%%%%%%%%%%%%%%%%%%%%%%%%
 \end{enumerate}
 %%%%%%%%%%%%%%%%%%%%%%%%%%%%
 
From the points 4.\,and 5.\,above, one obtains 
 an alternative, conceptual proof of the identity  \eqref{Eq:Main-equation} of
Theorem \ref{Thm:main}. This proof relies on the decompositions of 
 $\wedge^{r-2}\mathbf C^{\cL}$ in irreducible $W$-modules  
 which  are interesting on their own and appear to be new for $r>4$ (see  \cite[Proposition\,3.2]{PirioChat}). Note, for $r = 8$ one would need to adapt this
approach in order to prove our main result, possibly by considering the subrepresentation given
by $\wedge^{r-2}_{exc} \mathbf C^{\cL}$ (Lemma \ref{Lem:In-wedge-exc}).  
 More generally, it would be interesting to determine the decomposition into irreducibles of  
$\wedge^{r-2}_{exc} \mathbf C^{\cL}$ for any 
$r$ and to verify wether  it admits the signature as one of its irreducible components or not (we know that it is not the case for $r\leq 7$).

%%%%%%%%%%%%%%%%%%%%%%%%%%%%%%%%%%%%%%%
%%%%%%%%%%%%%%%%%%%%%%%%%%%%%%%%%%%%%%%
\subsection{The identity ${\bf HLog}^{r-2}$ is defined over $\ZZ$}
\label{SS:over-Z}
%%%%%%%%%%%%%%%%%%%%%%%%%%%%%%%%%%%%%%%
%%%%%%%%%%%%%%%%%%%%%%%%%%%%%%%%%%%%%%%
By requiring that all the residues considered are integers, one defines canonical $\ZZ$-structures  $\HH^\ZZ$, $\HH^\ZZ_{\Sigma_\mc}$, $\HH_{\mc}^\ZZ$
on the spaces $\HH$, $\HH_{\Sigma_\mc}$, $\HH_{\mc}$ respectively, which are compatible with respect to pull-backs and inclusions, i.e., one has 
$\phi_\mc^*\HH^\ZZ_{\Sigma_\mc}=\HH^\ZZ_{\mc}$ and $\HH_\mc\subset \HH$ induces an inclusion  $\HH_\mc^\ZZ \subset \HH^\ZZ$
for any $\mc\in \cK$.  Moreover, the residue map ${\rm Res}_L$ of \S\ref{SS:SS-32} admits a canonical lift $%{\rm Res}_L^\ZZ : 
 \HH^\ZZ\longrightarrow \ZZ^\cL$
over $\ZZ$. We leave it to the reader to verify that all the statements in \S\ref{SS:SS-33}--\S\ref{SS:representation-theoretic-interpretation} hold over $\ZZ$. In conclusion,  
${\bf HLog}^{r-2}$ is defined over $\ZZ$, a fact which may be interesting from an arithmetic perspective.

%%%%%%%%%%%%%%%%%%%%%%%%%%%%%%%%%%%%%%%
%%%%%%%%%%%%%%%%%%%%%%%%%%%%%%%%%%%%%%%
\subsection{} 
%%%%%%%%%%%%%%%%%%%%%%%%%%%%%%%%%%%%%%%
%%%%%%%%%%%%%%%%%%%%%%%%%%%%%%%%%%%%%%%
For any $n\geq 2$, it is tempting to consider more generally blow-ups $Y_r=\Bl_r\big(\PP^n\big)$ at $r\geq n+2$ general points and attempt to generalize Theorem \ref{Thm:main} by following the exact same approach as in this section. One defines a symmetric bilinear form $(\cdot , \cdot )$ on $K_{Y_r}^{\perp}\subset\Pic(Y_r)=\ZZ\{H,E_1\ldots, E_r\}$ (where $H$ is the hyperplane class and the $E_i$'s are the classes of the exceptional divisors) by setting
$$\big(H,H\big)=n-1\, , \qquad \big(H,E_i\big)=0\qquad \mbox{ and } \qquad  \big(E_i,E_j\big)=-\delta_{ij}\quad  \mbox{ for } \, i,j=1,\ldots,r\,.$$ 
One can define a Coxeter group $W$ associated to a T-shaped Dynkin-type diagram $T_{2,n+1,r-n-1}$ such that $W$ acts on $\Pic(Y_r)$ in a geometric fashion, in particular preserving the bilinear form $(\cdot , \cdot )$. It is known that $W$ is finite if and only if 
 \begin{equation}
 \label{Eq:nr} 
\frac{1}{n+1}+\frac{1}{r-n-1}>\frac{1}{2}\,.
\end{equation}
 This condition is equivalent to $Y_r$ being a \emph{Mori dream space} \cite{Mukai, CT} and translates to $r\leq n+3$ if $n\geq5$, $r\leq 8$ if $n=4, 2$, and $r\leq 7$ if $n=3$. Such blow-ups are natural generalizations of del Pezzo surfaces. We refer to \cite{Mukai, CT} for more details. 

Assume that  \eqref{Eq:nr}   is satisfied.  One can replace the set of lines $\cL$ by the set of \emph{Weyl divisors}, i.e., divisors in the finite orbit $W\cdot E_1$, and the set $\cK$ of conic classes with the set of \emph{Weyl pencils}, which we define as one-dimensional linear systems $|E+F|$, for $E, F$ Weyl divisors such that 
$(E,F)=1$. Such a linear system $\mc$ induces a rational map $\phi_\mc : Y_r\dashrightarrow \PP^1$ which has 
$ r-n+1$ reducible fibers, with components $E', F'$ Weyl divisors such that $(E',F')=1$ and $E'+F'=E+F$ in $\Pic(Y_r)$. It  is straightforward to check that the group $W$ acts transitively on the sets $\cL$ and $\cK$. All the constructions leading up to Lemma \ref{Lem:Lolo} hold in this more general context, i.e., for each Weyl pencil $\mc\in \cK$ we can construct canonical elements $\tau_\mc\in \wedge^{r-n}\CC^\cL$ (well-defined up to a global sign) such that 
Lemma \ref{Lem:Lolo} and the statements 1.\,and 2.\,at the beginning of \S\ref{SS:zogo} hold. 
However, if $n\geq3$, the identity $\sum_{\mc}\tau_\mc=0$ never holds. One can follow the same approach as in this section to prove an analogue of Lemma \ref{Claim:E-exceptional-(r-2)-tuple}: for any \emph{exceptional $(r-n)$-tuple} $\mathcal E$ there are exactly $n$ Weyl pencils such that each element of $\mathcal E$ appears as a component of a reducible fiber of the associated fibration. The analogue of Lemma \ref{Lem:gno-gno} is that 
the coefficient with which an element $E_1\wedge\ldots\wedge E_{r-n}$ in the exceptional basis appears in $\sum_{\mc\in\cK}\tau_\mc$ is (up to a sign) $(n-2)$, hence, never zero if $n\geq3$.  
In analytic terms, this translates as the fact that, if $AI^{r-n}_\mc$ stands for the hyperlogarithm on $\mathbf P^1$ such that 
$AI^{r-n}_\mc\big( \phi_\mc \big) ={\rm II}_{Y_r}^y( \tau_\mc)$ for any $\mc\in \cK$ (for a previously chosen base point $y$ general in $ Y_r$), then the functional identity $\sum_{\mc\in\cK} AI^{r-n}_\mc\big( \phi_\mc \big)=0$ is not satisfied in the vicinity of $y$ on $Y_r$.

However, the space ${\bf HLog}^{r-n}_{Y_r}$
of tuples $(\alpha_\mc)_{\mc \in \cK}\in \mathbf C^{\cK}$ such that 
$\sum_{\mc\in\cK} \alpha_\mc\, AI^{r-n}_\mc\big( \phi_\mc \big) =0$ is not trivial. 
Similar to the case of del Pezzo surfaces, for every exceptional $r$-tuple 
$\mathscr E$ there exists a small modification $F_{\mathscr E}: Y_r\dashrightarrow  {\rm Bl}_{q_1,\ldots,q_r} (\mathbf P^n)$, where 
$\beta: {\rm Bl}_{q_1,\ldots,q_r} (\mathbf P^n) \rightarrow \mathbf P^n$ is a blow-up of a (possibly distinct) configuration of points 
$q_1,\ldots,q_r$ in general position. Let $J\subset \{1,\ldots,r\}$ be of cardinal $n-2$ and 
let $\pi_J: \mathbf P^n\dashrightarrow\mathbf P^2$ be the  linear projection from the 
$(n-3)$-plane in $\mathbf P^n$ spanned by the $q_j$'s for $j\in J$. Setting 
$q_k'=\pi_J(q_k)$ for $k\not \in J$, one has $r-n+2$ points in general position in $\PP^2$. 
The total space of the blow-up $\beta_J: {\rm Bl}_{\{q_k'\}_{k\notin J}}\big(\mathbf P^2\big)\rightarrow \mathbf P^2$ is a del Pezzo  surface of degree $d=7-r+n$ which 
we will denote by ${\rm dP}_{d,J}$. We let $\Pi_{{\mathscr E},J}: Y_r\dashrightarrow {\rm dP}_{d,J}$ be the induced rational maps (so 
that one has 
$\beta_J\circ \Pi_{{\mathscr E},J}=\pi_J\circ \beta\circ F_{\mathscr E}$ as rational maps). If $\cK_J$ denotes the set of conic classes on 
${\rm dP}_{d,J}$ and $\psi_{\boldsymbol{\kappa}} : {\rm dP}_{d,J}\rightarrow \mathbf P^1$ denotes the associated conic fibration for 
${\boldsymbol{\kappa}}\in \cK_J$, then the compositions $\phi_{\boldsymbol{\kappa}} = \psi_{\boldsymbol{\kappa}}\circ \Pi_{{\mathscr E},J}: Y_r\dashrightarrow \mathbf P^1$ are Weyl pencils (and all Weyl pencils are of this form). 
Consequently one obtains an injection $
\cK_J \subset \cK$. 
If 
  $\sum_{ {\boldsymbol{\kappa}}\in \cK_J} \epsilon_{\boldsymbol{\kappa}} \,
AI^{r-n}_ {\boldsymbol{\kappa}}\big( \psi_ {\boldsymbol{\kappa}} \big) =0$ 
  stands for the 
identity  on ${\rm dP}_{d,J}$ given by Theorem 
\ref{Thm:MMain} (with $\epsilon_{\boldsymbol{\kappa}} \in \{\pm 1\}$ for any ${\boldsymbol{\kappa}}\in \cK_J $),  one gets that the hyperlogarithmic identity $\sum_{ {\boldsymbol{\kappa}}\in \cK_J} \epsilon_{\boldsymbol{\kappa}} \,
AI^{r-n}_ {\boldsymbol{\kappa}}\big( \phi_ {\boldsymbol{\kappa}} \big) =0$ holds true locally at $y$ on $Y_r$. 
This identity corresponds to a non-zero element of ${\bf HLog}^{r-n}_{Y_r}$, which we denote by 
 ${\bf HLog}^{r-n}_{\mathscr E,J}$.   It follows that ${\bf HLog}^{r-n}_{Y_r}$ is not trivial. 
 
It is natural to ask 
whether the span of the set of ${\bf HLog}^{r-n}_{\mathscr E,J}$'s  for all pairs $(\mathscr E,J)$ as above coincides with the whole space ${\bf HLog}^{r-n}_{Y_r}$ or not. 
For the case when $r=n+2$ with $n\geq 2$ arbitrary, this follows from computations in \cite{Pereira}. 
By direct computations, we have verified that it is the case as well for $(n,r)=(3,6)$ and $(n,r)=(4,7)$.  We conjecture that this happens in all cases. If true, this would say that regarding functional identities satisfied by the complete antisymmetric hyperlogarithms $AI^{r-n}_\mc$ on $Y_r$, there is nothing new since  everything come from the 2-dimensional del Pezzo hyperlogarithmic identity ${\bf HLog}^{r-n}$ 
(up to pull-backs under the maps $\Pi_{\mathscr E,J} : Y_r\dashrightarrow {\rm dP}_{d,J}$).

%%%%%%%%%%%%%%%%%%%%%%%%%%%%%%%%%%%%%%%
%%%%%%%%%%%%%%%%%%%%%%%%%%%%%%%%%%%%%%%
\section{\bf The identity ${\bf HLog}^3$ in explicit form}
\label{S:HLog3}
%%%%%%%%%%%%%%%%%%%%%%%%%%%%%%%%%%%%%%%
%%%%%%%%%%%%%%%%%%%%%%%%%%%%%%%%%%%%%%%
%%%%%%%%%%%%%%%%%%%%%%%%%%%%%%%%%%%%%%%
%%%%%%%%%%%%%%%%%%%%%%%%%%%%%%%%%%%%%%%
The identity ${\bf HLog}^2$ is equivalent to Abel's relation
$\boldsymbol{\big(\mathcal Ab\big)}$
 which is written in explicit form.  One can make the other hyperlogarithmic identities ${\bf HLog}^{r-2}$ explicit as well. We illustrate this with the case when $r=5$.\sk

Let $X_5$  stand  for the blow-up of $\mathbf P^2$ at the following five points:   $p_1=[1:0:0]$, $p_2=[0:1:0]$, $p_3=[0:0:1]$,  $p_4=[1:1:1]$ and $p_5=[a : b : 1]$, 
for some parameters $a,b\in \mathbf C$ such that
%%%%%%%%%%%%%%%%%%%%%%%%%%%%%%%%%%%%%%%
\begin{equation*}
%\label{Eq:PG-pi-gamma}
%%%%%%%%%%%%%%%%%%%%%%%%%%%%%%%%%%%%%%%
ab(a-1)(b-1)(a-b)\neq 0\, ,
%%%%%%%%%%%%%%%%%%%%%%%%%%%%%%%%%%%%%%%
\end{equation*}
%%%%%%%%%%%%%%%%%%%%%%%%%%%%%%%%%%%%%%%
a condition that we assume to be satisfied in what follows.  

We make explicit the  weight 3 hyperlogarithmic identity ${\bf H Log}(X_5)$ when expressed in the 
affine coordinates $x,y$ corresponding to the affine embedding $\mathbf C^2\hookrightarrow \mathbf P^2$, $(x,y)\mapsto [x:y:1]$.

Relatively to the coordinates $x,y$, the conic fibrations on $X_5$ correspond  on $\mathbf P^2$ to the following rational functions $U_i$ 
(where $P$ stands for the affine polynomial $P=(1-b )x - (1-a )y - (a - b)$)
: 
$$
\scalebox{1.2}{
\begin{tabular}{lllll}
$U_1= x$ & $ U_2=  \frac1y$ 
& $ U_3=  \frac{y}{x} $ 
& $U_4= 
\frac{x-y}{x-1} $ 
& $  U_5= 
   \frac{b(a-x ) }{a y - b x} $ 
  \vspace{0.25cm}  \\
 $ U_6=    \frac{
 P}{ (x - 1)(y-b)}$ 
   & $ U_7=     \frac{ (x - y)(y-b)}{ y\,P}$ 
    &
$    U_8=  \frac{x\,P}{ (x - y)(x - a)} $
     & 
$ U_9= \,
          \frac{  
      y(x - a)}{ x (y-b)}  $ 
      & 
      $U_{10}=  \frac{
       x(y - 1)}{y (x - 1)} \, .$
\end{tabular}
}
$$

For any $i=1,\ldots,10$, the set of $\lambda\in \PP^1$ 
for which $U_i^{-1}(\lambda)$ is reducible  has the form $\{ 0,1, {\mathfrak r}_i, \infty\}$ where ${\mathfrak r}_i\in \mathbf P^1\setminus \{ 0,1, \infty \}$ is given by
%%%%%%%%%%%%%%%%%%%%%%%%%%%%%%%%%%%%
\begin{align*}
%\label{Eq:WdP4-Ui}
   \nonumber
{\mathfrak r}_1= & \,  a  && {\mathfrak r}_2=  \frac1b
&& {\mathfrak r}_3=  \frac{b}{a} 
&& {\mathfrak r}_4= 
\frac{a-b}{a-1} 
&&  {\mathfrak r}_5= 
   \frac{b(a-1 ) }{a  - b }  \\
 {\mathfrak r}_6=  &\,   \frac{
 b-a}{ b}
   && {\mathfrak r}_7= 
          \frac{ 1}{1-a}
    &&
    {\mathfrak r}_8= 1-b 
     && 
{\mathfrak r}_9= \,
          \frac{  
     1 - a}{ 1-b} 
      && 
      {\mathfrak r}_{10}=  \frac{
       a(b - 1)}{b (a - 1)} \, .
           \nonumber
%%%%%%%%%%%%%%%%%%%%%%%%%%
\end{align*} 
%%%%%%%%%%%%%%%%%%%%%%%%%%%%%%%%%%%%

For a triple $(a,b,c)$ of pairwise distinct points on $\mathbf C$ and a given base point $\xi \in \mathbf C\setminus \{a,b,c\}$, we consider the  weight 3 hyperlogarithm  $L_{a,b,c}^{\xi}$ 
defined  by
$$ 
L_{a,b,c}^{\xi}(z)= \int_{\xi}^{z}   \Bigg(\int_{\xi}^{u_3} \bigg(\int_{\xi}^{u_2} \frac{du_1}{u_1-c}\bigg) \frac{du_2}{u_2-b}\Bigg)  \frac{du_3}{u_3-a} $$
for any $z$ sufficiently close to $\xi$, and we denote by $AI_{a,b,c}^{\xi}$ its antisymmetrization:
\begin{equation}
\label{Eq:AI-a,b,c}
AI_{a,b,c}^{\xi}=\frac{1}{6} \, \bigg(\, L_{a,b,c}^{\xi}-L_{a,c,b}^{\xi}-L_{b,a,c}^{\xi}+L_{b,c,a}^{\xi}+L_{c,a,b}^{\xi}-L_{c,b,a}^{\xi}
\bigg) 
\, .
\end{equation}

We now fix a base point $\zeta\in \mathbf C^2$ image of a point in $X_5$ which does not belong to any line. For $i=1,\ldots,10$, we set  $\zeta_i=U_i(\zeta)\in \mathbf C\setminus \{0,1,r_i\}$  and $$ AI_i^3=AI^{\zeta_i}_{0,1,r_i}\, . $$ 
  Then one can verify that 
${\bf H Log}(X_5)$  has the following explicit form
 \begin{equation}
\label{Eq:Eq-explicit-d=4}
\sum_{i=1}^{10} AI_i^3\big(U_i\big)=0\, , 
\end{equation}
a functional identity which is  satisfied on any sufficiently small neighborhood of $\zeta$.

%%%%%%%%%%%%%%%%%%%%%%%%%%%%%%%%%%%%%%%%
%%%%%%%%%%%%%%%%%%%%%%%%%%%%%%%%%%%%%%%

\bigskip
%\bigskip

%\newpage 
%${}^{}$

\vfill 
{\small  ${}^{}$ \hspace{-0.6cm} {\bf Ana-Maria Castravet, Luc Pirio}\\
Laboratoire de Math\'ematiques de Versailles\\
 Universit\'e Paris-Saclay, UVSQ \& CNRS  (UMR 8100)\\
 45 Avenue des \'Etats-Unis, 78000 Versailles, France\\
 E-mails: \href{mailto:ana-maria.castravet@uvsq.fr}{\tt ana-maria.castravet@uvsq.fr}, \href{mailto:luc.pirio@uvsq.fr}{\tt luc.pirio@uvsq.fr}

%Universit\'e Paris-Saclay, UVSQ, Laboratoire de Math\'ematiques de Versailles, 
%45 Avenue des \'Etats Unis, 78035 Versailles, France }

\end{document}